\documentclass[a4paper]{amsart}

\usepackage{verbatim}
\usepackage{amssymb}
\usepackage{amscd}
\usepackage{amsmath}
\usepackage{psfrag}

\usepackage[dvips]{graphicx}
\usepackage{psfrag}

\theoremstyle{plain}
\newtheorem{thm}{Theorem}[section]
\newtheorem{cor}[thm]{Corollary}
\newtheorem{prop}[thm]{Proposition}
\newtheorem{lem}[thm]{Lemma}

\theoremstyle{definition}
\newtheorem{dfn}[thm]{Definition}

\theoremstyle{remark}
\newtheorem{rem}[thm]{Remark}

\makeatletter
 
 \@addtoreset{equation}{section}
\makeatother

\newcommand{\R}{\mathbb{R}}

\newcommand{\Z}{\mathbb{Z}}

\newcommand{\Int}{\operatorname{Int}}
\newcommand{\id}{\operatorname{id}}
\newcommand{\co}{\colon\thinspace}
\newcommand{\Image}{\mathop{\mathrm{Im}}\nolimits}

\title[Singular fibers of stable maps of $3$-manifolds with boundary]
{Singular fibers of stable maps of 
$3$-manifolds with boundary into surfaces and 
their applications}

\author[Osamu Saeki]{Osamu Saeki}
\author[Takahiro Yamamoto]{Takahiro Yamamoto}

\address{Institute of Mathematics for Industry,
Kyushu University,
Motooka 744, Nishi-ku, Fukuoka 819-0395, Japan}
\email{saeki@imi.kyushu-u.ac.jp}

\address{Faculty of Engineering, 
Kyushu Sangyo University, 
3-1 Matsukadai 2-chome, Higashi-ku,
Fukuoka, 813-8503, Japan}

\email{yama.t@ip.kyusan-u.ac.jp}



\begin{document}

\begin{abstract}
In this paper, we first classify singular fibers
of proper $C^\infty$ stable maps of $3$-dimensional
manifolds with boundary into surfaces.
Then, we compute the cohomology groups
of the associated universal complex of singular fibers,
and obtain certain cobordism invariants for
Morse functions on compact surfaces with boundary. 
\end{abstract}

\maketitle

\section{Introduction}

Let $M$ and $N$ be smooth manifolds,
where $M$ may possibly have boundary, while
$N$ has no boundary.
For a $C^\infty$ map
$f\co M \to N$ and
a point $q \in N$,
the pre-image $f^{-1}(q)$ is called
the \emph{level set} of $f$ over $q$.
Throughout this paper, we call
the map germ along the level set
\[
f\co (M, f^{-1}(q)) \to (N, q)
\]
the \emph{fiber} over $q$, adopting the terminology
introduced in \cite{Saeki04}.
Furthermore, 
if a point $q \in N$ is a regular value of both
$f$ and $f|_{\partial M}$, then we call the fiber (or
the level set) over $q$ 
a \emph{regular fiber} (resp.\ a \emph{regular
level set});
otherwise, a \emph{singular fiber} (resp.\ a
\emph{singular level set}). 

Some equivalence relations
among fibers are defined as follows.
Let $f_i\co M_i \to N_i$, $i = 0, 1$, be $C^\infty$ maps. 
For $q_i \in N_i$, $i =0, 1$, we say that the fibers over 
$q_0$ and $q_1$ are \emph{$C^\infty$ equivalent} 
(or \emph{$C^0$ equivalent}) if 
for some open neighborhoods $U_i$ of $q_i$ in $N_i$, 
there exist diffeomorphisms (resp.\ homeomorphisms)
$\Phi \co f_0^{-1}(U_0) \to f_1^{-1}(U_1)$ and 
$\varphi \co U_0 \to U_1$ with $\varphi(q_0) =q_1$ 
which make the following diagram commutative:
\[
\begin{CD}
 (f_0^{-1}(U_0), f_0^{-1}(q_0)) @> \Phi>> 
(f_1^{-1}(U_1), f_1^{-1}(q_1)) \\
    @V{f_0}VV  @VV{f_1}V\\
        (U_0,q_0) @> \varphi>> (U_1,q_1).
\end{CD}
\]

Denote by $C^\infty(M, N)$ the set of $C^\infty$ maps $M \to N$ 
equipped with the Whitney $C^\infty$ topology. 
A $C^\infty$ map $f\co M \to N$ is called a 
\emph{$C^\infty$ stable map}
(or a \emph{$C^0$ stable map})
if there exists a neighborhood $N(f) \subset C^\infty(M, N)$ 
of $f$ such that every map $g \in N(f)$ is $C^\infty$ 
right-left equivalent (resp.\ $C^0$ right-left equivalent)
to $f$ \cite{GolubitskyGuillemin}. 
Here, two maps $f$ and $g \in C^\infty(M, N)$ are 
\emph{$C^\infty$ right-left equivalent} (or \emph{$C^0$ 
right-left equivalent})
if there exist diffeomorphisms (resp.\ homeomorphisms)
$\Psi \co M \to M$ 
and $\psi \co N \to N$ such that 
$f\circ \Psi =\psi \circ g$. 
When we just say that $f$ is a \emph{stable map}, it will
mean that it is a $C^\infty$ stable map.

The notion of singular fibers of $C^\infty$ maps between
manifolds without boundary was first
introduced in \cite{Saeki04}, where classifications
of singular fibers of stable maps $M \to N$ with
$(\dim{M}, \dim{N}) = (2, 1), (3, 2)$ and $(4, 3)$
were obtained.
Later, singular fibers
of stable maps of manifolds without boundary were studied 
in \cite{Saeki04, Saeki06, SaekiY06, SaekiY07, Y06, Y07, Y08},
especially in connection with cobordisms.
The first author \cite{Saeki04} established the theory 
of universal complex of singular fibers of $C^\infty$ maps
as an analogy of the Vassiliev complex for map germs \cite{Ohmoto, V}.
This can be used for getting certain cobordism
invariants of singular maps. 
For example, 
the first author~\cite{Saeki04} 
obtained cobordism invariants for stable Morse functions 
on closed surfaces, and 
the second author \cite{Y08} studied the
universal complex of singular fibers of two-colored $C^\infty$ maps,
computing its cohomology groups.
In these theories, 
for a certain set of singular fibers $\tau$, 
cohomology classes of the universal complexes of 
singular fibers of $\tau$-maps provide $\tau$-cobordism invariants
for $\tau$-maps. 

In this paper, we study singular fibers of proper $C^\infty$ 
stable maps of $3$-dimensional manifolds with boundary into surfaces
without boundary. 
By using such fibers, we obtain cobordism invariants for 
stable Morse functions on compact surfaces with boundary. 

The paper is organized as follows. 
In \S2, we classify fibers of proper $C^\infty$ stable maps 
of $3$-dimensional manifolds with boundary into surfaces
without boundary,
with respect to the $C^\infty$ equivalence relation. 
For this we use several known results on the classification
of stable singularities of maps on manifolds with boundary 
\cite{MartinsNabarro13, Shibata00} together
with the techniques developed by the first author in \cite{Saeki04}.
In \S3, we obtain several co-existence formulae of singular fibers 
of $C^\infty$ stable maps of compact $3$-dimensional
manifolds with boundary into surfaces without boundary. These formulae
can be obtained by analyzing the adjacencies of the singular
fibers. 
In \S4, we construct the universal complex of 
singular fibers of proper $C^\infty$ stable maps of 
$3$-dimensional manifolds with boundary into surfaces
without boundary. 
By carefully computing the cohomology groups of the universal complex, 
we obtain certain cobordism invariants for 
stable Morse functions on compact surfaces with boundary. 

Throughout the paper, 
all manifolds and maps between them are smooth
of class $C^\infty$ unless otherwise stated. 
For a smooth map $f \co M \to N$ between manifolds,
we denote by $S(f)$ the set of points in $M$
where the differential of $f$ does not have
maximal rank $\min \{\dim{M}, \dim{N}\}$.
For a space $X$, 
$\id_X$ denotes the identity map of $X$. 
For a (co)cycle $c$, we denote by $[c]$ the (co)homology class 
represented by $c$. 

\section{Classification of singular fibers}

%


In this section, we classify singular fibers of proper
$C^\infty$ stable maps of $3$-dimensional manifolds
with boundary into surfaces without boundary.

%

We can prove the following characterization of $C^\infty$ stable maps 
$f\co M \to N$ by using standard techniques in singularity
theory together with the results on local normal forms obtained
in \cite{MartinsNabarro13, Shibata00}.

\begin{prop}[Shibata~\cite{Shibata00}, 
Martins and Nabarro~\cite{MartinsNabarro13}]\label{prop:stable}
Let $M$ be a $3$-manifold possibly with boundary and
$N$ a surface without boundary. 
A proper $C^\infty$ map $f\co M \to N$ is 
$C^\infty$ stable if and only if it satisfies the following conditions. 
\begin{enumerate}
\item
\textup{(}Local conditions\textup{)} 
In the following, for $p \in \partial M$,
we use local coordinates
$(x, y, z)$ around $p$ such that $\Int{M}$ and $\partial M$
correspond to the sets $\{z > 0\}$ and $\{z = 0\}$, respectively.
\begin{itemize}
\item[(1a)]
For $p \in \Int{M}$, 
the germ of $f$ at $p$ is right-left equivalent to one of the 
following:
\[
(x, y, z) \mapsto 
\begin{cases}
(x, y), & \text{$p$: regular point,}
\\
(x, y^2 +z^2), & \text{$p$: definite fold point,}
\\
(x, y^2 -z^2), & \text{$p$: indefinite fold point,}
\\
(x, y^3 + xy - z^2), & \text{$p$: cusp point.}
\end{cases}
\]
\item[(1b)]
For $p \in \partial M \setminus S(f)$, 
the germ of $f$ at $p$ is right-left equivalent to one of the 
following:
\[
(x, y, z) \mapsto
\begin{cases}
(x, y), & \text{$p$: regular point of $f|_{\partial M}$,}
\\
(x, y^2 +z), & \text{$p$: boundary definite fold point,}
\\
(x, y^2 -z), & \text{$p$: boundary indefinite fold point,}
\\
(x, y^3 + xy +z), & \text{$p$: boundary cusp point.}
\end{cases}
\]
\item[(1c)]
For $p \in \partial M \cap S(f)$, 
the germ of $f$ at $p$ is right-left equivalent to the
map germ
\[
(x, y, z) \mapsto 
(x, y^2 +xz \pm z^2).
\]
\end{itemize}
\item
\textup{(}Global conditions\textup{)} 
For each $q \in f(S(f)) \cup f(S(f|_{\partial M}))$, 
the multi-germ 
\[
(f|_{S(f) \cup S(f|_{\partial M})}, 
f^{-1}(q) \cap (S(f) \cup S(f|_{\partial M})))
\]
is right-left equivalent to one of the
eight multi-germs whose images
are depicted in Figure~\textup{\ref{localMto2}},
where the ordinary curves correspond to 
$f(S(f))$ and the dotted curves to
$f(S(f|_{\partial M}))$: 
$(1)$ and $(4)$ represent immersion mono-germs 
$(\R, 0) \ni t \mapsto (t, 0) \in (\R^2, 0)$ which 
correspond to a single fold point and 
a single boundary fold point respectively, 
$(3)$, $(6)$ and $(7)$ represent normal crossings of 
two immersion germs, each of which corresponds to a fold point
or a boundary fold point, 
$(2)$ and $(5)$ represent cusp mono-germs 
$(\R, 0) \ni t \mapsto (t^2, t^3) \in (\R^2, 0)$ which 
correspond to a cusp point and 
a boundary cusp point respectively, 
and $(8)$ represents the 
restriction of the mono-germ \textup{(1c)},
corresponding to a single point in $\partial M \cap S(f)$,
to the singular point set.
\begin{figure}[tbhp]
\psfrag{q}{$q$}
\psfrag{a}{$(1)$}
\psfrag{b}{$(2)$}
\psfrag{c}{$(3)$}
\psfrag{d}{$(4)$}
\psfrag{e}{$(5)$}
\psfrag{f}{$(6)$}
\psfrag{g}{$(7)$}
\psfrag{h}{$(8)$}
\includegraphics[scale=.7]{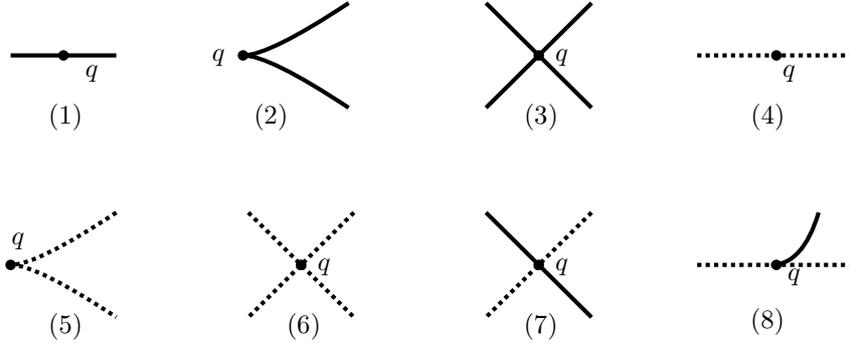}
\caption{The images of multi-germs of $f|_{S(f) \cup S(f|_{\partial M})}$}
\label{localMto2}
\end{figure}
\end{enumerate}
\end{prop}

Note that if a $C^\infty$ map $f\co M \to N$ is $C^\infty$ stable, 
then so is $f|_{\partial M} \co \partial M \to N$. 

In the following, a point $p$ on the boundary such that
the map germ at $p$ is
right-left equivalent
to the normal form
$$(x, y, z) \mapsto (x, y^2 +xz +z^2) \quad
\mbox{\rm or} \quad
(x, y, z) \mapsto (x, y^2 +xz -z^2)$$ 
is called 
a \emph{definite $\Sigma^{2, 0}_{1, 0}$ point} or
an \emph{indefinite $\Sigma^{2, 0}_{1, 0}$ point}, respectively. 


%

Let $M_i$ be smooth manifolds and $A_i \subset M_i$ 
be subsets, $i =0, 1$. 
A continuous map $g \colon A_0 \to A_1$ is said to be 
\emph{smooth}
if for every $q \in A_0$, there exists a 
smooth map $\widetilde{g}\colon V\to M_1$ defined on a 
neighborhood $V$ of $q \in M_0$ such that 
$\widetilde{g}|_{V \cap A_0} =g|_{V \cap A_0}$. 
Furthermore, a smooth map $g\colon A_0 \to A_1$ is called
a \emph{diffeomorphism} if it is a homeomorphism and its inverse is also smooth. 
When there exists a diffeomorphism between $A_0$ and $A_1$, 
we say that they are \emph{diffeomorphic}. 

Then, by Proposition~\ref{prop:stable}, we have the following
local descriptions on singular level sets. In the statement,
it may not be necessary to distinguish some of the 
cases if only the topology of level sets is
concerned; nevertheless, we divide the cases
as below in order to introduce symbols
which take corresponding map germs into
account.

\begin{lem}\label{lemma2.2}
Let $M$ be a $3$-manifold possibly with boundary, $N$
a surface without boundary,
and $f\co M \to N$ a proper $C^\infty$ stable map.
Then, every point $p \in S(f) \cup S(f|_{\partial M})$
has one of the following 
neighborhoods in its corresponding singular level set 
\textup{(}see Figures~\textup{\ref{LocalFibers}} and
\textup{\ref{BFibers}} for reference\textup{)}: 
\begin{enumerate}
\item
isolated point diffeomorphic to 
$\{(y, z) \in \R^2\,|\, y^2 +z^2 =0\}$, 
if $p \in \Int{M}$ is a definite fold point, 
\item
union of two transverse arcs diffeomorphic to 
$\{(y, z) \in \R^2 \,|\, y^2 -z^2 =0\}$, 
if $p \in \Int{M}$ is an indefinite fold point, 
\item
cuspidal arc diffeomorphic to 
$\{(y, z) \in \R^2 \,|\, y^3 -z^2 =0\}$,
if $p \in \Int{M}$ is a cusp point, 
\item
isolated point diffeomorphic to 
$\{(y, z) \in \R^2 \,|\, y^2 +z =0,\, z \geq 0\}$,
if $p \in \partial M$ is a boundary definite fold point, 
\item
arc diffeomorphic to 
$\{(y, z) \in \R^2 \,|\, y^2 -z =0,\, z \geq 0\}$, 
if $p \in \partial M$ is a boundary indefinite fold point, 
\item
arc diffeomorphic to 
$\{(y, z) \in \R^2 \, |\, y^3 +z =0,\, z \geq 0\}$,
if $p \in \partial M$ is a boundary cusp point, 
\item
isolated point diffeomorphic to 
$\{(y, z) \in \R^2 \, |\, y^2 +z^2 =0,\, z \geq 0\}$, 
if $p \in \partial M \cap S(f)$ is a definite $\Sigma^{2, 0}_{1, 0}$ point, 
\item
polygonal line diffeomorphic to 
$\{(y, z) \in \R^2 \, |\, y^2 -z^2 =0,\, z \geq 0\}$, 
if $p \in \partial M \cap S(f)$ is an indefinite $\Sigma^{2, 0}_{1, 0}$ point. 
\end{enumerate}
\end{lem}

\begin{figure}[tbhp]
\psfrag{(1)}{$(1)$}
\psfrag{(2)}{$(2)$}
\psfrag{(3)}{$(3)$}
\psfrag{(4)}{$(4)$}
\psfrag{(5)}{$(5)$}
\psfrag{(6)}{$(6)$}
\psfrag{(7)}{$(7)$}
\psfrag{(8)}{$(8)$}
\includegraphics{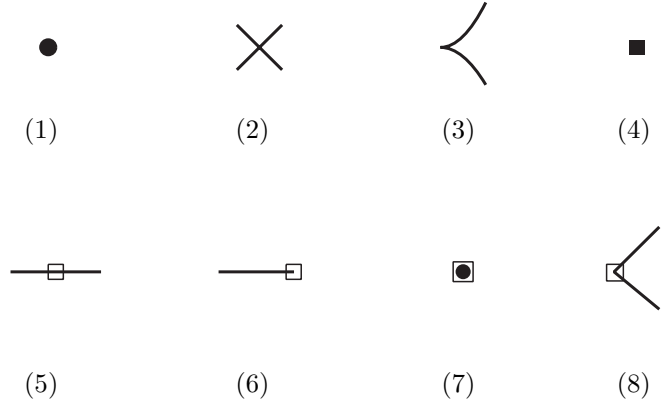}
\caption{Neighborhoods of singular points in their singular level sets}
\label{LocalFibers}
\end{figure}

\begin{figure}[tbhp]
\psfrag{4}{$(4)$}
\psfrag{5}{$(5)$}
\psfrag{6}{$(6)$}
\psfrag{7}{$(7)$}
\psfrag{8}{$(8)$}
\includegraphics[width=0.95\linewidth,height=0.3\textheight,
keepaspectratio]{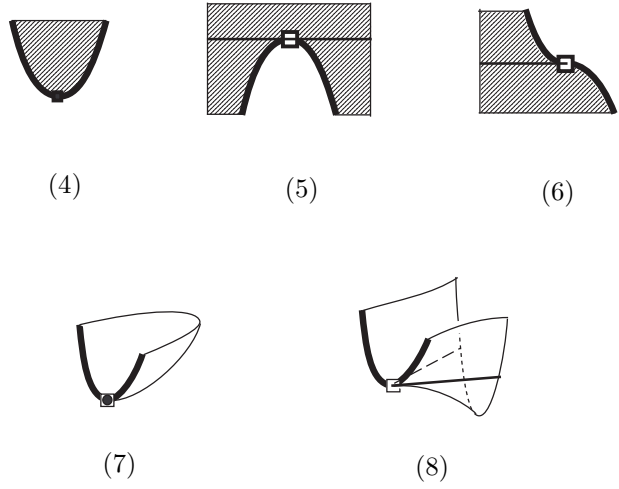}
\caption{Singular level sets touching the boundary}
\label{BFibers}
\end{figure}

\begin{rem}\label{rem:dot}
Note that in Figure~\ref{LocalFibers}, 
the black dot $(1)$, the black square $(4)$, and the
black dot surrounded by a square $(7)$ all represent an isolated point; 
however, we use distinct symbols in order to distinguish their
corresponding map germs.
In the figures, the squares represent
points on the boundary; more precisely, they
are points in $S(f|_{\partial M})$.
\end{rem}


In Figure~\ref{BFibers}, singular level sets that intersect
the boundary are depicted, where the surfaces 
appearing in the figures correspond
to the hypersurface $x=0$, the intersection of the
hypersurface with the boundary is depicted by
thick curves, and the associated map germs restricted
to the hypersurface correspond
to the respective height functions.
In fact, for the cases (4) and (5),
the relevant map germs of $3$-manifolds into surfaces
are right-left equivalent to the suspensions
of the function germs
in the sense of Definition~\ref{def:suspension}.
As to the cases (6), (7) and (8), the relevant map germs
are obtained by using certain ``generic deformations''
of the function germs.

For the local nearby level sets, 
we have the following, which can be proved by
direct calculations using the corresponding normal forms.

\begin{lem}\label{lemma2.3}
Let $M$ be a $3$-manifold possibly with boundary, $N$
a surface without boundary,
$f\co M \to N$ a proper $C^\infty$ stable map,
and 
$p \in S(f) \cup S(f|_{\partial M})$ a singular point 
such that 
$f^{-1}(f(p)) \cap (S(f) \cup S(f|_{\partial M})) =\{p\}$. 
Then, the level sets 
near $p$ are as depicted in Figure~\textup{\ref{NearbyFibers}}:
\begin{enumerate}
\item
$p \in \Int{M}$ is a definite fold point, 
\item
$p \in \Int{M}$ is an indefinite fold point, 
\item
$p \in \Int{M}$ is a cusp point, 
\item
$p \in \partial M$ is a boundary definite fold point, 
\item
$p \in \partial M$ is a boundary indefinite fold point, 
\item
$p \in \partial M$ is a boundary cusp point, 
\item
$p \in S(f) \cap \partial M$ is a 
definite $\Sigma^{2, 0}_{1, 0}$ point, 
\item
$p \in S(f) \cap \partial M$ is an 
indefinite $\Sigma^{2, 0}_{1, 0}$ point,
\end{enumerate}
where each of the $0$-dimensional objects and the
thin $1$-dimensional objects represents
a portion of the level set over the corresponding point 
in the target,
each of the thick curves represents
$f(S(f))$, and each of the dotted curves represents 
$f(S(f|_{\partial M}))$ near $f(p)$.
Furthermore, the dotted squares represent \textup{(}transverse\textup{)}
intersections with $\partial M$.
\end{lem}

\begin{figure}[tbhp]
\psfrag{e}{$\emptyset$}
\psfrag{(1)}{$(1)$}
\psfrag{(2)}{$(2)$}
\psfrag{(3)}{$(3)$}
\psfrag{(4)}{$(4)$}
\psfrag{(5)}{$(5)$}
\psfrag{(6)}{$(6)$}
\psfrag{(7)}{$(7)$}
\psfrag{(8)}{$(8)$}
\includegraphics{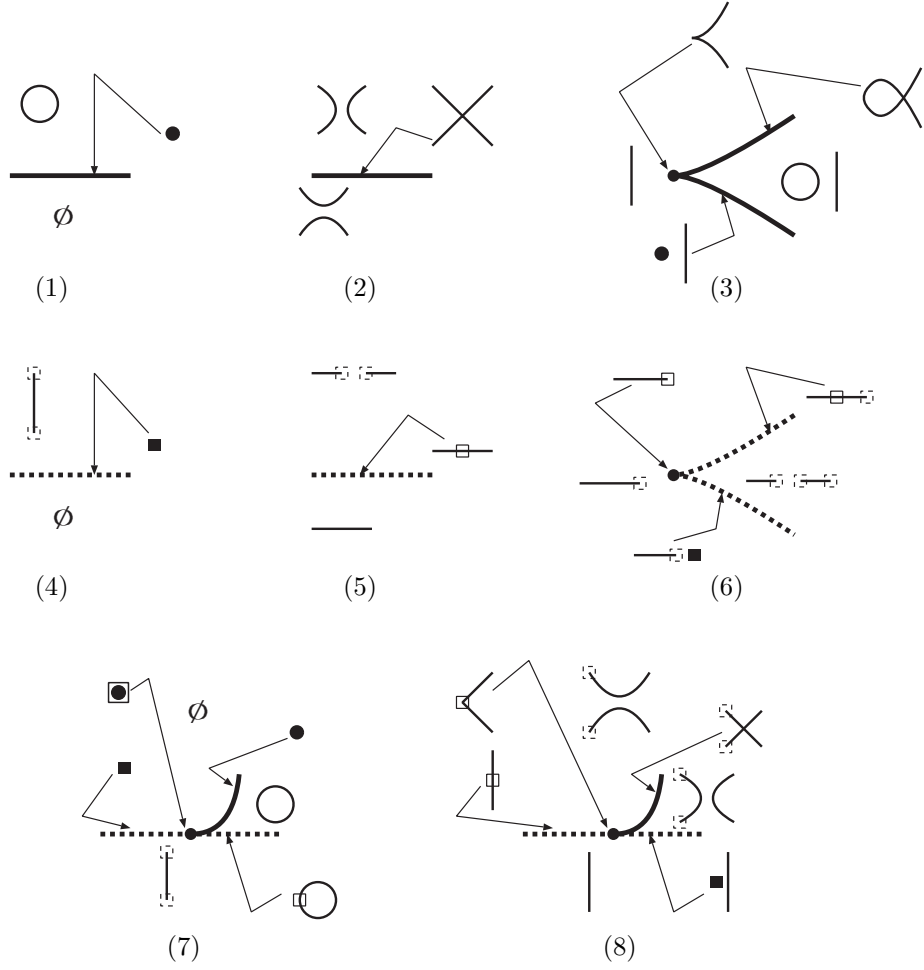}
\caption{Local degenerations of level sets}
\label{NearbyFibers}
\end{figure}

\begin{dfn}
Suppose that we are given a finite number of
fibers of smooth maps, where all the dimensions
of the sources and the targets are the same.
Then, their \emph{disjoint union} is the fiber
corresponding to the single map defined on the
disjoint union of the sources, where the
target spaces are all identified to a single small
open disk. This definition clearly depends
on such identifications; however, in the
following, we can take ``generic identifications''
in such a way that the resulting fiber
is $C^\infty$ stable, and the result
is unique up to $C^\infty$ equivalence
as long as the above identifications are generic.
\end{dfn}

Now, by using the method developed in
\cite{Saeki04}, we get the following
list of singular fibers. 
We omit the proof here. 

\begin{prop}
\textup{(i)} Each diagram in Figures~$\ref{fibersbM2toR}$
and $\ref{fibersbM3toR2}$ uniquely determines
the $C^\infty$ equivalence class of fibers in such a way
that the diagram
represents the corresponding central level set, under the
convention described in Remark~$\ref{rem:dot}$.

\textup{(ii)} Let $f\co M \to N$ be a proper $C^\infty$ stable map of 
a $3$-manifold $M$ with boundary into a surface $N$
without boundary. 
Then, every fiber of $f$ is $C^\infty$ 
equivalent to the disjoint union of 
one of the fibers in the following list, a finite number of 
copies of a fiber of the trivial circle bundle, and a finite
number of copies of a fiber of the trivial $I$-bundle, where
$I = [-1, 1]$: 
\begin{enumerate}
\item
fibers as depicted in Figure~$\ref{fibersbM2toR}$, i.e.\ 
$\widetilde{\mathrm{b0}}^0$, $\widetilde{\mathrm{b0}}^1$, and
$\widetilde{\mathrm{bI}}^\mu$ with $2 \leq \mu \leq 10$,
\item
fibers
$\widetilde{\mathrm{bII}}^{\mu, \nu}$ with $2 \leq \mu
\leq \nu \leq 10$,
where $\widetilde{\mathrm{bII}}^{\mu, \nu}$
means the disjoint union of $\widetilde{\mathrm{bI}}^\mu$
and $\widetilde{\mathrm{bI}}^\nu$,
\item
fibers as depicted in Figure~$\ref{fibersbM3toR2}$, i.e.\ 
$\widetilde{\mathrm{bII}}^\mu$ with $11 \leq \mu \leq 39$,
$\widetilde{\mathrm{bII}}^{a}$,
$\widetilde{\mathrm{bII}}^{b}$,
$\widetilde{\mathrm{bII}}^{c}$,
$\widetilde{\mathrm{bII}}^{d}$,
$\widetilde{\mathrm{bII}}^{e}$ and
$\widetilde{\mathrm{bII}}^{f}$.
\end{enumerate}
More precisely, two such fibers containing
no singular points of the restriction to the boundary
are $C^\infty$ equivalent
if and only if their corresponding level sets are
diffeomorphic. Therefore, in the figures,
the associated level sets are depicted together with
the information on the corresponding local map germs
which are depicted in accordance with Lemmas~\textup{\ref{lemma2.2}}
and \textup{\ref{lemma2.3}}.
\label{ClassifyFibers}
\end{prop}

\begin{figure}[tbhp]
\psfrag{k=0}{$\kappa =0$}
\psfrag{k=1}{$\kappa =1$}
\psfrag{00}{$\widetilde{\mathrm{b0}}^0$}
\psfrag{01}{$\widetilde{\mathrm{b0}}^1$}
\psfrag{11}{$\widetilde{\mathrm{bI}}^1$}
\psfrag{12}{$\widetilde{\mathrm{bI}}^2$}
\psfrag{13}{$\widetilde{\mathrm{bI}}^3$}
\psfrag{14}{$\widetilde{\mathrm{bI}}^4$}
\psfrag{15}{$\widetilde{\mathrm{bI}}^5$}
\psfrag{16}{$\widetilde{\mathrm{bI}}^6$}
\psfrag{17}{$\widetilde{\mathrm{bI}}^7$}
\psfrag{18}{$\widetilde{\mathrm{bI}}^8$}
\psfrag{19}{$\widetilde{\mathrm{bI}}^9$}
\psfrag{10}{$\widetilde{\mathrm{bI}}^{10}$}
\centering
\includegraphics[%
width=0.95\linewidth,height=0.95\textheight,%
keepaspectratio]{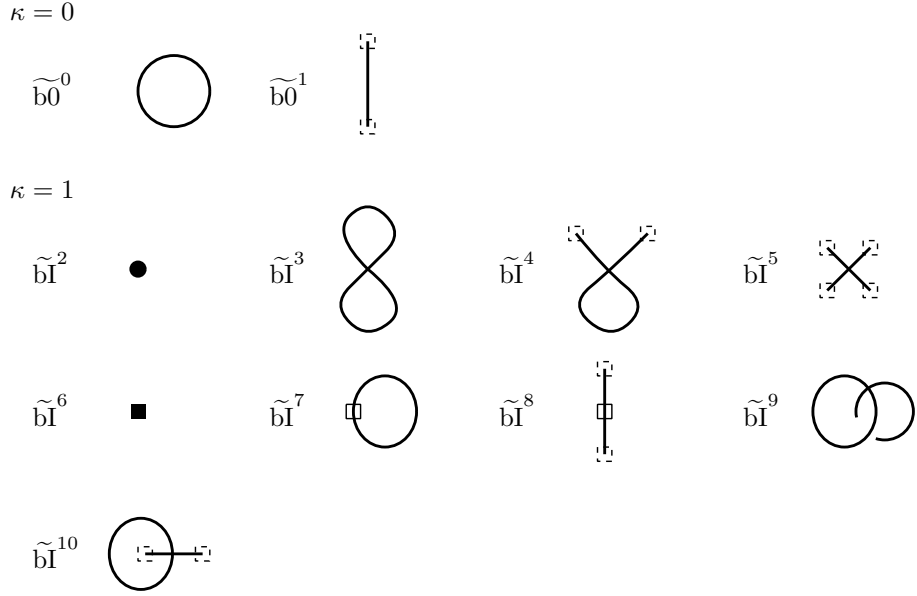}
\caption{List of the fibers of proper 
$C^\infty$ stable maps of $3$-manifolds with boundary 
into surfaces without boundary; $1$}
\label{fibersbM2toR}
\end{figure}

\begin{figure}[tbhp]
\psfrag{k=2}{$\kappa =2$}
\psfrag{211}{$\widetilde{\mathrm{bII}}^{11}$}
\psfrag{212}{$\widetilde{\mathrm{bII}}^{12}$}
\psfrag{213}{$\widetilde{\mathrm{bII}}^{13}$}
\psfrag{214}{$\widetilde{\mathrm{bII}}^{14}$}
\psfrag{215}{$\widetilde{\mathrm{bII}}^{15}$}
\psfrag{216}{$\widetilde{\mathrm{bII}}^{16}$}
\psfrag{217}{$\widetilde{\mathrm{bII}}^{17}$}
\psfrag{218}{$\widetilde{\mathrm{bII}}^{18}$}
\psfrag{219}{$\widetilde{\mathrm{bII}}^{19}$}
\psfrag{220}{$\widetilde{\mathrm{bII}}^{20}$}
\psfrag{221}{$\widetilde{\mathrm{bII}}^{21}$}
\psfrag{222}{$\widetilde{\mathrm{bII}}^{22}$}
\psfrag{223}{$\widetilde{\mathrm{bII}}^{23}$}
\psfrag{224}{$\widetilde{\mathrm{bII}}^{24}$}
\psfrag{225}{$\widetilde{\mathrm{bII}}^{25}$}
\psfrag{226}{$\widetilde{\mathrm{bII}}^{26}$}
\psfrag{227}{$\widetilde{\mathrm{bII}}^{27}$}
\psfrag{228}{$\widetilde{\mathrm{bII}}^{28}$}
\psfrag{229}{$\widetilde{\mathrm{bII}}^{29}$}
\psfrag{230}{$\widetilde{\mathrm{bII}}^{30}$}
\psfrag{231}{$\widetilde{\mathrm{bII}}^{31}$}
\psfrag{232}{$\widetilde{\mathrm{bII}}^{32}$}
\psfrag{233}{$\widetilde{\mathrm{bII}}^{33}$}
\psfrag{234}{$\widetilde{\mathrm{bII}}^{34}$}
\psfrag{235}{$\widetilde{\mathrm{bII}}^{35}$}
\psfrag{236}{$\widetilde{\mathrm{bII}}^{36}$}
\psfrag{237}{$\widetilde{\mathrm{bII}}^{37}$}
\psfrag{238}{$\widetilde{\mathrm{bII}}^{38}$}
\psfrag{239}{$\widetilde{\mathrm{bII}}^{39}$}
\psfrag{240}{$\widetilde{\mathrm{bII}}^{a}$}
\psfrag{241}{$\widetilde{\mathrm{bII}}^{b}$}
\psfrag{242}{$\widetilde{\mathrm{bII}}^{c}$}
\psfrag{243}{$\widetilde{\mathrm{bII}}^{d}$}
\psfrag{244}{$\widetilde{\mathrm{bII}}^{e}$}
\psfrag{245}{$\widetilde{\mathrm{bII}}^{f}$}
\centering
\includegraphics[%
width=0.95\linewidth,height=0.95\textheight,%
keepaspectratio]{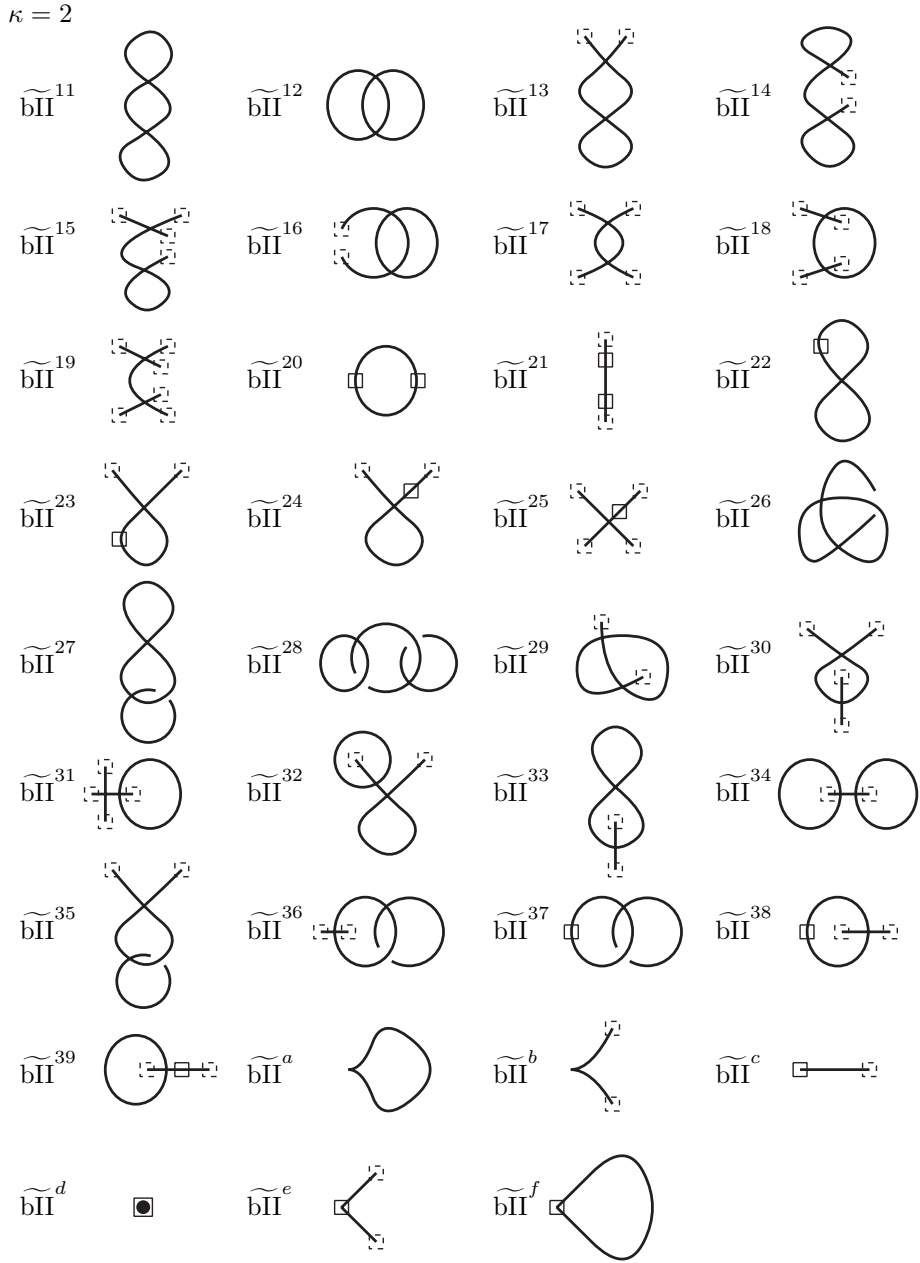}
\caption{List of the fibers of proper 
$C^\infty$ stable maps of $3$-manifolds with boundary 
into surfaces without boundary; $2$}
\label{fibersbM3toR2}
\end{figure}

In Figures~\ref{fibersbM2toR} and $\ref{fibersbM3toR2}$, 
$\kappa$ denotes the \emph{codimension} of the set of points 
in the target $N$ whose corresponding fibers are $C^\infty$ equivalent 
to the relevant one 
(see~\cite{Saeki04} for details). 
Furthermore, the symbols $\widetilde{\mathrm{b0}}^\ast$,
$\widetilde{\mathrm{bI}}^\ast$, and
$\widetilde{\mathrm{bII}}^\ast$ 
mean the names of the 
corresponding fibers. 
Note that we have named the
fibers so that each fiber with connected central level set has its
own number or letter, and a fiber with disconnected
central level set 
has the name consisting of the numbers
of its ``connected components''. 

Note that the fibers whose central level sets
do not intersect $S(f|_{\partial M})$
have been essentially obtained in \cite[\S6]{Saeki06}.

We can prove
Proposition~\ref{ClassifyFibers} by using the relative version 
of Ehresmann's fibration theorem together with
Proposition~\ref{prop:stable}.
See \cite[Proof of Theorem~3.5]{Saeki04} for details. 

Then, we immediately obtain the following corollary. 
(For details, see \cite[Proof of Corollary~3.9]{Saeki04}.)

\begin{cor}
Two fibers of proper $C^\infty$ stable maps
of $3$-manifolds with boundary into surfaces without boundary
are $C^\infty$ equivalent 
if and only if they are
$C^0$ equivalent. 
\end{cor}

\begin{rem}
If the source $3$-manifold is orientable, 
then the singular fibers of types 
$\widetilde{\mathrm{bI}}^9$, $\widetilde{\mathrm{bI}}^{10}$, 
$\widetilde{\mathrm{bII}}^{26}$, $\widetilde{\mathrm{bII}}^{27}$, 
$\widetilde{\mathrm{bII}}^{28}$, $\widetilde{\mathrm{bII}}^{29}$, 
$\widetilde{\mathrm{bII}}^{30}$, $\widetilde{\mathrm{bII}}^{31}$, 
$\widetilde{\mathrm{bII}}^{32}$, $\widetilde{\mathrm{bII}}^{33}$, 
$\widetilde{\mathrm{bII}}^{34}$, $\widetilde{\mathrm{bII}}^{35}$, 
$\widetilde{\mathrm{bII}}^{36}$, $\widetilde{\mathrm{bII}}^{37}$, 
$\widetilde{\mathrm{bII}}^{38}$, and $\widetilde{\mathrm{bII}}^{39}$ 
never appear. 
\end{rem}

\begin{rem}
Our classification result of singular fibers
of stable maps of compact orientable
$3$-dimensional manifolds with boundary
into surfaces
has already been applied in computer science.
More precisely, it helps to visualize
characteristic features of certain multi-variate data.
For details, see \cite{ST2014, STSWKCDY2013}.
\end{rem}

\begin{rem}\label{rem:Morse}
Let $V$ be a surface with boundary and $W$ be
the real line $\R$ or the circle $S^1$.
A proper $C^\infty$ function $f\co V \to W$ is 
$C^\infty$ stable (i.e.\ it is a \emph{stable Morse function}) 
if and only if it satisfies the following conditions. 
\begin{enumerate}
\item
\textup{(}Local conditions\textup{)}
In the following, for $p \in \partial V$,
we use local coordinates
$(x, y)$ around $p$ such that $\Int{V}$ and $\partial V$
correspond to the sets $\{y > 0\}$ and $\{y = 0\}$, respectively.
\begin{itemize}
\item[(1a)]
For $p \in \Int{V}$, 
the germ of $f$ at $p$ is right-left equivalent to one of the following:
\[
(x, y) \mapsto 
\begin{cases}
x, & \text{$p$: regular point,}
\\
x^2 \pm y^2, & \text{$p$: fold point or non-degenerate critical point, }
\end{cases}
\]
\item[(1b)]
For $p \in \partial V$, 
the germ of $f$ at $p$ is right-left equivalent to one of the 
following:
\[
(x, y) \mapsto
\begin{cases}
x, & \text{$p$: regular point of $f|_{\partial V}$, }
\\
x^2 \pm y,  & \text{$p$: boundary fold point or}
\\ & \qquad \text{non-degenerate
critical point of $f|_{\partial V}$.}
\end{cases}
\]
\end{itemize}
\item
(Global conditions)
$f(p_1) \ne f(p_2)$ if $p_1 \ne p_2 \in S(f) \cup S(f|_{\partial V})$.
\end{enumerate}
The list of the $C^\infty$ equivalence classes of singular fibers of 
proper stable Morse functions on surfaces with boundary 
can be obtained in a similar fashion. The result
corresponds to those 
appearing in Figure~$\ref{fibersbM2toR}$ with 
$\kappa =0, 1$. In fact, it is not difficult to show that the suspensions
of the fibers of such functions
in the sense of Definition~\ref{def:suspension}
coincide with those appearing in the figure. However, in the
following, by abuse of notation, 
we use the symbols in Figure~$\ref{fibersbM2toR}$ with 
$\kappa =0, 1$ for the fibers of stable 
Morse functions as well.
\end{rem}

\section{Co-existence of singular fibers}

Let $M$ be a compact $3$-dimensional manifold with boundary, $N$
a surface without boundary,
$f\co M \to N$ a $C^\infty$ stable map, and 
$\widetilde{\mathcal{F}}$ 
a $C^\infty$ 
equivalence class of singular fibers of codimension $\geq 1$. 
Define $\widetilde{\mathcal{F}}(f)$ to be the set of points $q \in N$ 
such that the fiber over $q$ is $C^\infty$ equivalent to the disjoint union of 
$\widetilde{\mathcal{F}}$ 
and some copies of a fiber of a trivial circle bundle and 
some copies of a fiber of a trivial $I$-bundle. 
Furthermore, define $\widetilde{\mathcal{F}}_o(f)$ 
(or $\widetilde{\mathcal{F}}_e(f)$) 
to be the subset of $\widetilde{\mathcal{F}}(f)$ which consists of the points 
$q \in N$ such that 
the number of regular fibers, 
namely the total number of $\widetilde{\rm b0}^0$ components and 
$\widetilde{\rm b0}^1$ components in the fiber, 
is odd (resp.\ even). For codimension zero fibers,
by convention, we denote by $\widetilde{\mathrm{b0}}_o(f)$ (or
$\widetilde{\mathrm{b0}}_e(f)$) the set of points
$q \in N$ over which lies a regular fiber consisting
of an odd (resp.\ even) number of components.

If $\widetilde{\mathcal{F}}$ is of
codimension $1$, then the closure of $\widetilde{\mathcal{F}}_o(f)$ 
(or $\widetilde{\mathcal{F}}_e(f)$) 
is a finite graph embedded in $N$. 
Its vertices correspond to points over which lies a singular fiber 
of codimension $2$. 

The handshake lemma of the classical graph theory 
implies the following formulae. In the following,
for a finite set $S$, $|S|$ denotes its cardinality.

\begin{prop}
Let $f\co M \to N$ be a $C^\infty$ stable map of a compact 
$3$-manifold $M$ with boundary into a surface $N$ without
boundary. 
Then, the following numbers are always even:
\begin{enumerate}
\item
$|\widetilde{\rm bII}^{2,3}(f)| +|\widetilde{\rm bII}^{2,4}(f)| 
+|\widetilde{\rm bII}^{2,6}(f)| +|\widetilde{\rm bII}^{2,8}(f)| 
+|\widetilde{\rm bII}^a_e(f)| +|\widetilde{\rm bII}^b_e(f)| 
\\
+|\widetilde{\rm bII}^d_o(f)|$, 
\item
$|\widetilde{\rm bII}^{2,3}(f)| +|\widetilde{\rm bII}^{2,4}(f)| 
+|\widetilde{\rm bII}^{2,6}(f)| +|\widetilde{\rm bII}^{2,8}(f)| 
+|\widetilde{\rm bII}^a_o(f)| +|\widetilde{\rm bII}^b_o(f)| 
\\
+|\widetilde{\rm bII}^d_e(f)|$, 
\item
$|\widetilde{\rm bII}^{2,3}(f)| +|\widetilde{\rm bII}^{3,4}(f)| 
+|\widetilde{\rm bII}^{3,6}(f)| +|\widetilde{\rm bII}^{3,8}(f)| 
+|\widetilde{\rm bII}^{13}_e(f)| +|\widetilde{\rm bII}^{22}_o(f)| 
\\
+|\widetilde{\rm bII}^a_o(f)|$, 
\item
$|\widetilde{\rm bII}^{2,3}(f)| +|\widetilde{\rm bII}^{3,4}(f)| 
+|\widetilde{\rm bII}^{3,6}(f)| +|\widetilde{\rm bII}^{3,8}(f)| 
+|\widetilde{\rm bII}^{13}_o(f)| +|\widetilde{\rm bII}^{22}_e(f)| 
\\
+|\widetilde{\rm bII}^a_e(f)|$, 
\item
$|\widetilde{\rm bII}^{2,4}(f)| +|\widetilde{\rm bII}^{3,4}(f)| 
+|\widetilde{\rm bII}^{4,6}(f)| +|\widetilde{\rm bII}^{4,8}(f)| 
+|\widetilde{\rm bII}^{13}_e(f)| +|\widetilde{\rm bII}^{22}_o(f)| 
\\
+|\widetilde{\rm bII}^{23}_o(f)| +|\widetilde{\rm bII}^{24}(f)| 
+|\widetilde{\rm bII}^b_o(f)| +|\widetilde{\rm bII}^f_o(f)|$, 
\item
$|\widetilde{\rm bII}^{2,4}(f)| +|\widetilde{\rm bII}^{3,4}(f)| 
+|\widetilde{\rm bII}^{4,6}(f)| +|\widetilde{\rm bII}^{4,8}(f)| 
+|\widetilde{\rm bII}^{13}_o(f)| +|\widetilde{\rm bII}^{22}_e(f)| 
\\
+|\widetilde{\rm bII}^{23}_e(f)| +|\widetilde{\rm bII}^{24}(f)| 
+|\widetilde{\rm bII}^b_e(f)| +|\widetilde{\rm bII}^f_e(f)|$, 
\item
$|\widetilde{\rm bII}^{2,5}(f)| +|\widetilde{\rm bII}^{3,5}(f)| 
+|\widetilde{\rm bII}^{4,5}(f)| +|\widetilde{\rm bII}^{5,6}(f)| 
+|\widetilde{\rm bII}^{5,8}(f)| +|\widetilde{\rm bII}^{15}(f)| 
\\
+|\widetilde{\rm bII}^{23}_o(f)| +|\widetilde{\rm bII}^{25}(f)|
+|\widetilde{\rm bII}^{30}_o(f)| +|\widetilde{\rm bII}^{38}_o(f)| 
+|\widetilde{\rm bII}^{e}_o(f)|$, 
\item
$|\widetilde{\rm bII}^{2,5}(f)| +|\widetilde{\rm bII}^{3,5}(f)| 
+|\widetilde{\rm bII}^{4,5}(f)| +|\widetilde{\rm bII}^{5,6}(f)| 
+|\widetilde{\rm bII}^{5,8}(f)| +|\widetilde{\rm bII}^{15}(f)| 
\\
+|\widetilde{\rm bII}^{23}_e(f)| +|\widetilde{\rm bII}^{25}(f)|
+|\widetilde{\rm bII}^{30}_e(f)| +|\widetilde{\rm bII}^{38}_e(f)| 
+|\widetilde{\rm bII}^{e}_e(f)|$, 
\item
$|\widetilde{\rm bII}^{2,6}(f)| +|\widetilde{\rm bII}^{3,6}(f)| 
+|\widetilde{\rm bII}^{4,6}(f)| +|\widetilde{\rm bII}^{6,8}(f)| 
+|\widetilde{\rm bII}^{c}_e(f)| +|\widetilde{\rm bII}^{d}_o(f)| 
\\
+|\widetilde{\rm bII}^{e}_e(f)| +|\widetilde{\rm bII}^{f}_e(f)|$, 
\item
$|\widetilde{\rm bII}^{2,6}(f)| +|\widetilde{\rm bII}^{3,6}(f)| 
+|\widetilde{\rm bII}^{4,6}(f)| +|\widetilde{\rm bII}^{6,8}(f)| 
+|\widetilde{\rm bII}^{c}_o(f)| +|\widetilde{\rm bII}^{d}_e(f)| 
\\
+|\widetilde{\rm bII}^{e}_o(f)| +|\widetilde{\rm bII}^{f}_o(f)|$, 
\item
$|\widetilde{\rm bII}^{2,7}(f)| +|\widetilde{\rm bII}^{3,7}(f)| 
+|\widetilde{\rm bII}^{4,7}(f)| +|\widetilde{\rm bII}^{6,7}(f)| 
+|\widetilde{\rm bII}^{7,8}(f)| +|\widetilde{\rm bII}^{22}(f)| 
\\
+|\widetilde{\rm bII}^{23}_e(f)| +|\widetilde{\rm bII}^{d}_o(f)| 
+|\widetilde{\rm bII}^{f}_o(f)|$, 
\item
$|\widetilde{\rm bII}^{2,7}(f)| +|\widetilde{\rm bII}^{3,7}(f)| 
+|\widetilde{\rm bII}^{4,7}(f)| +|\widetilde{\rm bII}^{6,7}(f)| 
+|\widetilde{\rm bII}^{7,8}(f)| +|\widetilde{\rm bII}^{22}(f)| 
\\
+|\widetilde{\rm bII}^{23}_o(f)| +|\widetilde{\rm bII}^{d}_e(f)| 
+|\widetilde{\rm bII}^{f}_e(f)|$, 
\item
$|\widetilde{\rm bII}^{2,8}(f)| +|\widetilde{\rm bII}^{3,8}(f)| 
+|\widetilde{\rm bII}^{4,8}(f)| +|\widetilde{\rm bII}^{6,8}(f)| 
+|\widetilde{\rm bII}^{23}_o(f)| +|\widetilde{\rm bII}^{24}(f)| 
\\
+|\widetilde{\rm bII}^{c}_o(f)| +|\widetilde{\rm bII}^{e}_o(f)|$, 
\item
$|\widetilde{\rm bII}^{2,8}(f)| +|\widetilde{\rm bII}^{3,8}(f)| 
+|\widetilde{\rm bII}^{4,8}(f)| +|\widetilde{\rm bII}^{6,8}(f)| 
+|\widetilde{\rm bII}^{23}_e(f)| +|\widetilde{\rm bII}^{24}(f)| 
\\
+|\widetilde{\rm bII}^{c}_e(f)| +|\widetilde{\rm bII}^{e}_e(f)|$, 
\item
$|\widetilde{\rm bII}^{2,9}(f)| +|\widetilde{\rm bII}^{3,9}(f)| 
+|\widetilde{\rm bII}^{4,9}(f)| +|\widetilde{\rm bII}^{6,9}(f)| 
+|\widetilde{\rm bII}^{8,9}(f)| +|\widetilde{\rm bII}^{27}(f)| 
\\
+|\widetilde{\rm bII}^{35}_e(f)| +|\widetilde{\rm bII}^{37}_o(f)|$, 
\item
$|\widetilde{\rm bII}^{2,9}(f)| +|\widetilde{\rm bII}^{3,9}(f)| 
+|\widetilde{\rm bII}^{4,9}(f)| +|\widetilde{\rm bII}^{6,9}(f)| 
+|\widetilde{\rm bII}^{8,9}(f)| +|\widetilde{\rm bII}^{27}(f)| 
\\
+|\widetilde{\rm bII}^{35}_o(f)| +|\widetilde{\rm bII}^{37}_e(f)|$, 
\item
$|\widetilde{\rm bII}^{2,10}(f)| +|\widetilde{\rm bII}^{3,10}(f)| 
+|\widetilde{\rm bII}^{4,10}(f)| +|\widetilde{\rm bII}^{6,10}(f)| 
+|\widetilde{\rm bII}^{8,10}(f)| +|\widetilde{\rm bII}^{30}_e(f)| 
\\
+|\widetilde{\rm bII}^{32}(f)| +|\widetilde{\rm bII}^{33}(f)| 
+|\widetilde{\rm bII}^{35}_o(f)| +|\widetilde{\rm bII}^{37}_o(f)| 
+|\widetilde{\rm bII}^{38}_o(f)| +|\widetilde{\rm bII}^{39}(f)|$, 
\item
$|\widetilde{\rm bII}^{2,10}(f)| +|\widetilde{\rm bII}^{3,10}(f)| 
+|\widetilde{\rm bII}^{4,10}(f)| +|\widetilde{\rm bII}^{6,10}(f)| 
+|\widetilde{\rm bII}^{8,10}(f)| +|\widetilde{\rm bII}^{30}_o(f)| 
\\
+|\widetilde{\rm bII}^{32}(f)| +|\widetilde{\rm bII}^{33}(f)| 
+|\widetilde{\rm bII}^{35}_e(f)| +|\widetilde{\rm bII}^{37}_e(f)| 
+|\widetilde{\rm bII}^{38}_e(f)| +|\widetilde{\rm bII}^{39}(f)|$.
\end{enumerate}
\end{prop}

By eliminating the terms of the forms 
$\mathcal{F}_o(f)$ and $\mathcal{F}_e(f)$, we obtain the following. 

\begin{cor}
Let $f\co M \to N$ be a $C^\infty$ stable map of a compact 
$3$-manifold $M$ with boundary into a
surface $N$ without boundary. 
Then, the following numbers are always even:
\begin{enumerate}
\item
$|\widetilde{\rm bII}^{a}(f)| +|\widetilde{\rm bII}^{b}(f)| 
+|\widetilde{\rm bII}^{d}(f)|$,
\item
$|\widetilde{\rm bII}^{13}(f)| +|\widetilde{\rm bII}^{22}(f)| 
+|\widetilde{\rm bII}^{a}(f)|$,
\item
$|\widetilde{\rm bII}^{13}(f)| +|\widetilde{\rm bII}^{22}(f)| 
+|\widetilde{\rm bII}^{23}(f)| +|\widetilde{\rm bII}^{b}(f)| 
+|\widetilde{\rm bII}^{f}(f)|$,
\item
$|\widetilde{\rm bII}^{23}(f)| +|\widetilde{\rm bII}^{30}(f)| 
+|\widetilde{\rm bII}^{38}(f)| +|\widetilde{\rm bII}^{e}(f)|$,
\item
$|\widetilde{\rm bII}^{c}(f)| +|\widetilde{\rm bII}^{d}(f)| 
+|\widetilde{\rm bII}^{e}(f)| +|\widetilde{\rm bII}^{f}(f)|$,
\item
$|\widetilde{\rm bII}^{23}(f)| +|\widetilde{\rm bII}^{d}(f)| 
+|\widetilde{\rm bII}^{f}(f)|$,
\item
$|\widetilde{\rm bII}^{23}(f)| +|\widetilde{\rm bII}^{c}(f)| 
+|\widetilde{\rm bII}^{e}(f)|$,
\item
$|\widetilde{\rm bII}^{35}(f)| +|\widetilde{\rm bII}^{37}(f)|$,
\item
$|\widetilde{\rm bII}^{30}(f)| +|\widetilde{\rm bII}^{35}(f)| 
+|\widetilde{\rm bII}^{37}(f)| +|\widetilde{\rm bII}^{38}(f)|$.
\end{enumerate}
\label{Co-existenceB3to2}
\end{cor}

\begin{rem}
The numbers in Corollary~$\ref{Co-existenceB3to2}$ are all even if and only if 
the following seven are all even:
\begin{enumerate}
\item
$|\widetilde{\rm bII}^{a}(f)| +|\widetilde{\rm bII}^{b}(f)| 
+|\widetilde{\rm bII}^{d}(f)|$,
\item
$|\widetilde{\rm bII}^{13}(f)| +|\widetilde{\rm bII}^{22}(f)|
+|\widetilde{\rm bII}^{a}(f)|$,
\item
$|\widetilde{\rm bII}^{23}(f)| +|\widetilde{\rm bII}^{e}(f)|$,
\item
$|\widetilde{\rm bII}^{c}(f)|$,
\item
$|\widetilde{\rm bII}^{d}(f)| +|\widetilde{\rm bII}^{e}(f)|
+|\widetilde{\rm bII}^{f}(f)|$,
\item
$|\widetilde{\rm bII}^{35}(f)| +|\widetilde{\rm bII}^{37}(f)|$,
\item
$|\widetilde{\rm bII}^{30}(f)| +|\widetilde{\rm bII}^{38}(f)|$.
\end{enumerate}
\end{rem}

\begin{rem}
By using the same method, we obtain similar co-existence results 
for singular fibers of a
stable Morse function $f\co V \to W$ of 
a compact surface $V$ with boundary into $W = \R$ or $S^1$:
\[|\widetilde{\rm bI}^2(f)| +|\widetilde{\rm bI}^3(f)| 
+|\widetilde{\rm bI}^4(f)| +|\widetilde{\rm bI}^6(f)| 
+|\widetilde{\rm bI}^8(f)| 
\equiv 0 \mod{2},
\]
where we are using the notation for the relevant
fibers in the sense of Remark~\ref{rem:Morse}.
\label{coexistence2to1_0}
\end{rem}

For a stable Morse function $f\co V \to W$ as above and 
a $C^\infty$ equivalence class $\widetilde{\mathcal{F}}$ of 
singular fibers, 
denote by $\widetilde{\mathcal{F}}_{o, o}(f)$ 
(or $\widetilde{\mathcal{F}}_{o, e}(f)$,
$\widetilde{\mathcal{F}}_{e, o}(f)$,
$\widetilde{\mathcal{F}}_{e, e}(f)$) 
the subset of $\widetilde{\mathcal{F}}(f)$ which consists of the 
points $q \in W$ such that the numbers of fibers of 
types $\widetilde{\rm b0}^0$ and 
$\widetilde{\rm b0}^1$ are both odd
(resp.\ odd and even, even and odd, or both even).
Then, by applying the same method to
the graph obtained as the closure of 
$\widetilde{\mathcal{F}}_{\ast, \star}(f)$ ($\ast, \star =o$ or $e$),
we obtain the following. 

\begin{lem}
Let $f\co V \to W$ be a stable Morse function on a compact 
surface $V$ with boundary into $W = \R$ or $S^1$. 
Then, the following numbers are always even:
\begin{enumerate}
\item
$|\widetilde{\rm bI}^2(f)| +|\widetilde{\rm bI}^3(f)| 
+|\widetilde{\rm bI}^4(f)| +|\widetilde{\rm bI}^7(f)|$, 
\item
$|\widetilde{\rm bI}^6(f)| +|\widetilde{\rm bI}^7(f)| 
+|\widetilde{\rm bI}^8(f)|$.
\end{enumerate}
\label{coexistence2to1}
\end{lem}

\section{Universal complex}

In this section we review the theory of 
universal complex of singular fibers and then
study the universal complex for stable maps of manifolds with boundary.
As to the general theory of universal complex
of singular fibers, refer to \cite[Chaps.~7 and 8]{Saeki04}.
 
Throughout this section, for a $C^\infty$ map $M \to N$, 
we assume that $M$ is an $m$-dimensional manifold 
which is not necessarily closed, and that 
$N$ is an $n$-dimensional manifold without boundary. 

In what follows, a \emph{codimension} of a smooth map
$f \co M \to N$ refers to the difference $\dim{N} -
\dim{M} \in \Z$.
To construct the universal complex of singular fibers 
of $C^\infty$ maps, we
fix an integer $\ell \in \Z$ for the codimension
of the maps, and consider
the following:
\begin{enumerate}
\item
a set $\tau$ of
fibers of proper Thom maps\footnote{
A \emph{Thom map} $M \to N$ is a $C^\infty$
stratified map with respect to Whitney 
regular stratifications of $M$ and $N$ such that 
it is a submersion on each stratum and 
satisfies certain regularity conditions.
See \cite[\S2]{Damon77}, \cite[\S2.5]{duPlessisWall95}, \cite{GWDL},
\cite[Part I, \S1]{Saeki04}, etc.\ for more details.}
of codimension $\ell$, and 
\item
an equivalence relation $\rho$ among the fibers in $\tau$. 
\end{enumerate}
We further assume that
the set $\tau$ and the relation $\rho$ satisfy 
the following conditions.
\begin{itemize}
\item[$(a)$] 
The set $\tau$ is closed under adjacency relation. 
That is, if a fiber is in $\tau$, then so are all nearby fibers. 
\item[$(b)$] 
The relation $\rho$ is weaker than the 
$C^0$ equivalence: 
each $\rho$-class is a union of $C^0$ equivalence classes. 
\item[$(c)$]
Let $f_i\co M_i \to N_i$ be proper 
Thom maps and $q_i \in N_i$, $i = 0, 1$. 
Suppose that 
the fibers over $q_0$ and $q_1$ are 
in $\tau$ and that they are equivalent with respect to $\rho$. 
Then, 
there exist open neighborhoods $U_i$ of $q_i$ in $N_i$, $i= 0, 1$, 
and a homeomorphism $\varphi \co U_0 \to U_1$ satisfying 
$\varphi(q_0)= q_1$ and 
$\varphi(U_0 \cap \mathcal{F}(f_0)) = 
U_1 \cap \mathcal{F}(f_1)$, 
for each $\rho$-class $\mathcal{F}$,
where for a proper Thom map $f \co M \to N$ of codimension $\ell$, we set
$$\mathcal{F}(f) = \{q \in N\,|\, \mbox{the fiber of $f$ over $q$
belongs to the class $\mathcal{F}$}\}.$$ 
\end{itemize}
In particular, the above conditions imply that for each proper 
Thom map $f \co M \to N$ and each $\rho$-class $\mathcal{F}$, 
$\mathcal{F}(f)$ 
is a $C^0$ submanifold of constant codimension 
unless it is not empty.
The \emph{codimension} of $\mathcal{F}$ is defined 
to be that of $\mathcal{F}(f)$ in $N$, and
is denoted by $\kappa(\mathcal{F})$. 


We call a proper Thom map $f\co M \to N$ 
a \emph{$\tau$-map} 
if all of its fibers are in $\tau$. 

For each $\kappa \in \Z$, let $C^{\kappa}(\tau, \rho)$ be the 
formal $\Z_2$-vector space 
spanned by the $\rho$-classes of codimension $\kappa$ in $\tau$. 
If there are no such fibers, then set 
$C^{\kappa}(\tau, \rho) = 0$.
Note that $C^{\kappa}(\tau, \rho)$ 
may possibly be of infinite dimension
in general.

Define the $\Z_2$-linear map 
$\delta_{\kappa} \co
C^{\kappa}(\tau,\rho) 
\to C^{\kappa + 1}(\tau,\rho)$ by 
\[
\delta_{\kappa}(\mathcal{F}) = 
\sum_{\kappa(\mathcal{G}) = \kappa + 1} 
n_{\mathcal{F}}(\mathcal{G}) \,
\mathcal{G}, 
\]
where 
$n_{\mathcal{F}}(\mathcal{G}) \in \Z_2$ 
is the number modulo two 
of the local components of $\mathcal{F}(f)$ 
which are locally adjacent to a component of
$\mathcal{G}(f)$ for a $\tau$-map $f$ satisfying 
$\mathcal{G}(f) \ne \emptyset$
(for details, see \cite{Saeki04}). 
Note that the coefficient 
$n_{\mathcal{F}}(\mathcal{G}) \in \Z_2$ 
is well-defined 
by condition $(c)$ for $\rho$. 

Since we see easily that
$\delta_{\kappa + 1} \circ \delta_{\kappa} = 0$ holds, 
we obtain the cochain complex 
\[
\mathcal{C}(\tau,\rho) = 
(C^{\kappa}(\tau,\rho), 
\delta_{\kappa})_{\kappa}. 
\] 
We call the resulting cochain complex 
the \emph{universal complex of singular fibers for proper 
$\tau$-maps with respect to the equivalence relation $\rho$}, 
and denote its cohomology group of dimension $\kappa$ by 
$H^{\kappa}(\tau,\rho)$. 

Let
$$
c = \sum_{\kappa(\mathcal{F}) = \kappa}n_{\mathcal{F}}\,\mathcal{F},
\quad n_{\mathcal{F}} \in \Z_2,
$$
be a $\kappa$-dimensional cochain of $\mathcal{C}(\tau,\rho)$. 
For a $\tau$-map 
$f\co M \to N$, 
denote by $c(f)$ the closure of the set of points 
$q \in N$ such that the fiber over $q$ is in 
$\mathcal{F}$ with $n_{\mathcal{F}} \ne 0$. 

If $c$ is a cocycle, then we can show that
$c(f)$ is a $\Z_2$-cycle of 
closed support of codimension $\kappa$ in $N$. 
In addition, 
if $M$ is compact and $\kappa > 0$, then
$c(f)$ is a $\Z_2$-cycle in the usual sense. 

\begin{lem}\label{cocycle_cohomology}
Let $c$ and $c'$ be $\kappa$-dimensional 
cocycles of $\mathcal{C}(\tau, \rho)$. 
If they are cohomologous, then 
the $\Z_2$-cycles $c(f)$ and $c'(f)$ 
are $\Z_2$-homologous in $N$ for 
each $\tau$-map $f\co M \to N$. 
\end{lem}

\begin{proof}
There exists a $(\kappa -1)$-dimensional 
cochain $d$ of $\mathcal{C}(\tau, \rho)$ 
such that 
$c - c' = \delta_{\kappa -1} d$. 
Then, we see easily that
$c(f) - c'(f) = \partial d(f)$ holds, and 
the result follows immediately.
\end{proof}

Let $[c]$ be the $\kappa$-dimensional cohomology 
class of $\mathcal{C}(\tau, \rho)$ represented
by a cochain $c$. 
For a $\tau$-map $f\co M \to N$, 
define $[c(f)] \in H^c_{n-\kappa}(N; \Z_2)$ 
to be the $\Z_2$-homology class 
represented by the $\Z_2$-cycle $c(f)$ of closed support,
where $H^c_*$ denotes the Borel--Moore homology or the
homology with closed support.
This is well-defined 
by virtue of Lemma~\ref{cocycle_cohomology}.

Define the $\Z_2$-linear map
\[
\varphi_{f}\co H^{\kappa}(\tau,\rho) 
\to H^{\kappa}(N; \Z_2)
\]
by $\varphi_{f}([c]) = [c(f)]^*$, 
where $[c(f)]^* \in H^{\kappa}(N; \Z_2)$ is the Poincar\'{e} 
dual of $[c(f)] \in H^c_{n -\kappa}(N; \Z_2)$. 

The suspension of a Thom map is introduced as follows.

\begin{dfn}\label{def:suspension}
For a proper Thom map $f\co M \to N$, 
consider the proper Thom map
$$
f \times \id_\R \co M \times \R \to N \times \R. 
$$
Note that $S(f \times \id_\R) = S(f) \times \R$ and 
$(f \times \id_\R)(S(f) \times \R) = f(S(f)) \times \R$.
We call $f\times \id_\R$ 
and the fiber of 
$f \times \id_\R$ over a point 
$(q,0) \in N \times \R$ 
the \emph{suspension} of $f$ and 
the \emph{suspension} of the fiber of 
$f$ over $q \in N$, respectively. 
\end{dfn}

Let $\tau$ be a set of fibers for proper
Thom maps of codimension $\ell$.
For a dimension pair $(m, n)$ with $n-m=\ell$,
let $\tau(m, n)$ denote the set of fibers
in $\tau$ for proper Thom maps of manifolds
of dimension $m$ into those of dimension $n$.
The equivalence relation on $\tau(m, n)$
induced by $\rho$ is denoted by $\rho_{m, n}$.

In the following, 
in addition to conditions $(a)$--$(c)$ above, we
assume the following two additional conditions.
\begin{itemize}
\item[$(d)$]
The suspension of each fiber in $\tau(m, n)$ is also 
in $\tau(m+1, n+1)$. 
\item[$(e)$]
If two fibers in $\tau(m, n)$ are equivalent with respect to 
$\rho_{m, n}$, then their suspensions are also 
equivalent with respect to $\rho_{m+1, n+1}$. 
\end{itemize}

For each $\kappa \in \Z$, the suspension induces 
the $\Z_2$-linear map
\[
s_{\kappa}\co C^{\kappa}(\tau(m+1, n+1), \rho_{m+1, n+1}) 
\to 
C^{\kappa}(\tau(m, n), \rho_{m, n}),
\]
where 
$s_{\kappa}(\mathcal{F})$ 
is the (possibly infinite) sum of all $\rho_{m, n}$-classes 
of codimension $\kappa$ 
whose suspensions are in $\mathcal{F}$.
Note that $s_{\kappa}$ is well-defined. 

\begin{lem}
The system of $\Z_2$-linear maps $\{s_{\kappa}\}$ 
defines the cochain map 
$$
\{s_{\kappa}\}\co 
\mathcal{C}(\tau(m+1, n+1), \rho_{m+1, n+1}) 
\to 
\mathcal{C}(\tau(m, n), \rho_{m, n}). 
$$
\end{lem}

We omit the proof of the above lemma.
For details, see \cite{Saeki04}.
 
Let us now introduce a geometric equivalence relation for $\tau$-maps. 

\begin{dfn}\label{d_cob}
Two $\tau$-maps $f_i\co M_i \to N$, $i= 0, 1$, 
of compact manifolds with boundary into a manifold without boundary 
are \emph{$\tau$-cobordant} 
if there exist a compact manifold $X$ 
with corners and 
a smooth map $F\co X \to N \times [0,1]$ that 
satisfy the following conditions:
\begin{enumerate}
\item
$\partial X =M_0 \cup Q \cup M_1$, where $M_0$, $M_1$
and $Q$ are codimension $0$ smooth submanifolds
of $\partial X$, $M_0 \cap M_1 =\emptyset$, and 
$\partial Q =(M_0 \cap Q) \cup (M_1 \cap Q)$, 
\item
$X$ has corners along $\partial Q$, 
\item
$F|_{M_0 \times [0, \varepsilon)} =f_0 \times \id_{[0, \varepsilon)}$ 
and 
$F|_{M_1 \times (1-\varepsilon, 1]} 
=f_1 \times \id_{(1-\varepsilon, 1]}$, 
where $M_0 \times [0, \varepsilon)$ and 
$M_1 \times (1-\varepsilon, 1]$ denote the collar neighborhoods 
(with corners) of 
$M_0$ and $M_1$ in $X$, respectively, 
\item  $F^{-1}(N \times \{i\}) =N_i$, $i =0, 1$, and, 
the restriction
$F|_{X \setminus (M_0 \cup M_1)} \co
X \setminus (M_0 \cup M_1) \to N \times (0, 1)$ 
is a $\tau$-map.
\end{enumerate}
In this case, we call the map $F$ a \emph{$\tau$-cobordism} between 
$f_0$ and $f_1$. 
\end{dfn}

Note that the $\tau$-cobordism relation is an 
equivalence relation among the $\tau$-maps into a fixed manifold $N$. 
For a manifold $N$, we denote
by $Cob_{\tau}(N)$
the set of all equivalence classes
of $\tau$-maps of compact manifolds into $N$ 
with respect to the $\tau$-cobordism.

\begin{prop}\label{key-lem}
If two $\tau$-maps $f_i\co M_i \to N$ 
of compact manifolds into $N$, $i=0,1$, are $\tau$-cobordant, 
then for each $\kappa$ we have 
$$
\varphi_{f_0}|_{\Image s_{\kappa \ast}} = 
\varphi_{f_1}|_{\Image s_{\kappa \ast}}\co 
\Image s_{\kappa \ast} \to H^{\kappa}(N;\Z_2), 
$$
where $s_{\kappa \ast}\co H^{\kappa}(\tau(m+1, n+1), \rho_{m+1, n+1}) 
\to H^{\kappa}(\tau(m, n), \rho_{m, n})$ denotes 
the homomorphism 
induced by the suspension. 
\end{prop}

\begin{proof}
Let $F\co X \to N \times [0,1]$ be a 
$\tau$-cobordism between $f_0$ 
and $f_1$. 
For each $\kappa$-dimensional cocycle 
$c =
\sum_{\kappa(\mathcal{F}) =\kappa} n_\mathcal{F}\mathcal{F}$ of 
$\mathcal{C}(\tau(m+1, n+1), \rho_{m+1, n+1})$, 
put $\overline{c} = s_{\kappa}(c) \in 
C^{\kappa}(\tau(m, n), \rho_{m, n})$. 
Then, we have
$$
\partial c(F) = 
\overline{c}(f_1) \times \{1\} 
- \overline{c}(f_0) \times \{0\},$$ 
since $c$ is a cocycle, 
where $c(F)$ denotes the closure in $N \times [0, 1]$
of the set of points 
$(q, t) \in N \times (0, 1)$ such that the fiber over 
$(q, t)$ is in $\mathcal{F}$ with $n_\mathcal{F} \ne 0$. 
Then the result follows immediately. 
\end{proof}

Thus, for each cohomology class
$[c] \in H^{\kappa}(\tau(m+1, n+1), \rho_{m+1, n+1})$
and an $n$-dimensional manifold $N$ without boundary,
we obtain the map
\[
I_{[c]} \co Cob_{\tau}(N) \to H^{\kappa}(N;\Z_2)
\] 
defined by $I_{[c]}(f) = \varphi_{f}({s_{\kappa}}_*[c])$,
which does not depend on the choice of a representative
$f$ of a given $\tau$-cobordism class. 
Namely, each element in 
$$H^{\kappa}(\tau(m+1, n+1), \rho_{m+1, n+1})$$
induces a $\tau$-cobordism invariant for
$\tau$-maps into an $n$-dimensional manifold $N$.

\subsection{Universal complex for stable maps of $n$-dimensional
manifolds with boundary into $(n-1)$-dimensional manifolds}

For a positive integer $n$, let
$b\mathcal{S}_{\mathrm{pr}}(n, n-1)$
be the set of fibers for proper $C^0$ stable Thom maps
of $n$-dimensional manifolds with boundary into $(n-1)$-dimensional
manifolds without boundary. We put
$$b\mathcal{S}_{\mathrm{pr}} = 
\bigcup_{n=1}^\infty b\mathcal{S}_{\mathrm{pr}}(n, n-1).$$

\begin{rem}
If the dimension pair $(n, n-1)$ is in the nice range,
then $C^0$ stable maps are $C^\infty$ stable (for
example, see \cite{duPlessisWall95}), and consequently
they are Thom maps. For example,
this is the case if $n \leq 8$.
\end{rem}

Furthermore,
let $\rho_{n, n-1}(2)$ be the \emph{$C^0$ equivalence relation modulo two 
regular fibers}
for fibers in $b\mathcal{S}_{\mathrm{pr}}(n, n-1)$: i.e., two fibers
in $b\mathcal{S}_{\mathrm{pr}}(n, n-1)$
are $\rho_{n, n-1}(2)$-equivalent if they
become $C^0$-equivalent after we add some regular
fibers to each of them with
the numbers of added components having the same parity.
Note that, under this equivalence, for $n = 2, 3$,
we do not distinguish the fibers of types 
$\widetilde{\rm b0}^0$ with
$\widetilde{\rm b0}^1$. Therefore, in the following,
we denote both of them by $\widetilde{\rm b0}$.

We denote by $\rho(2)$ the equivalence relation
on $b\mathcal{S}_{\mathrm{pr}}$ induced by $\rho_{n, n-1}(2)$,
$n \geq 1$. Note that the set $b\mathcal{S}_{\mathrm{pr}}$ 
and the equivalence relation $\rho(2)$ 
satisfy conditions (a)--(e) described above.

For a $C^0$ equivalence class $\widetilde{\mathcal{F}}$ of 
singular fibers, 
denote by $\widetilde{\mathcal{F}}_o$ (or $\widetilde{\mathcal{F}}_e$) 
the equivalence class with respect to $\rho_{n, n-1}(2)$ 
which consists of singular fibers of type $\widetilde{\mathcal{F}}$ 
with an odd number (resp.\ even number) of regular fiber components. 
For $n = 2, 3$, we denote by $\widetilde{\rm b0}_o$ and $\widetilde{\rm b0}_e$
the equivalence class with respect to $\rho_{n, n-1}(2)$ 
which consist exclusively of an odd (resp.\ even)
number of regular fiber components.

Then, for the universal complex 
$\mathcal{C}(b\mathcal{S}_{\mathrm{pr}}(3, 2), \rho_{3, 2}(2))$, 
the coboundary homomorphism is given by the 
following formulae, which are obtained
with the help of Lemma~\ref{lemma2.3}: 
\begin{eqnarray*}
\delta_0(\widetilde{\rm b0}_o) 
& = & \delta_0(\widetilde{\rm b0}_e) 
=\widetilde{\rm bI}^2 +\widetilde{\rm bI}^3 +\widetilde{\rm bI}^4 
+\widetilde{\rm bI}^6 +\widetilde{\rm bI}^8, 
\\ 
\delta_1(\widetilde{\rm bI}^2_o) 
& = & \widetilde{\rm bII}^{2,3} +\widetilde{\rm bII}^{2,4} 
+\widetilde{\rm bII}^{2,6} +\widetilde{\rm bII}^{2,8} 
+\widetilde{\rm bII}^a_e +\widetilde{\rm bII}^b_e 
+\widetilde{\rm bII}^d_o, 
\\ 
\delta_1(\widetilde{\rm bI}^2_e) 
& = & \widetilde{\rm bII}^{2,3} +\widetilde{\rm bII}^{2,4} 
+\widetilde{\rm bII}^{2,6} +\widetilde{\rm bII}^{2,8} 
+\widetilde{\rm bII}^a_o +\widetilde{\rm bII}^b_o 
+\widetilde{\rm bII}^d_e,
\\ 
\delta_1(\widetilde{\rm bI}^3_o) 
& = & \widetilde{\rm bII}^{2,3} +\widetilde{\rm bII}^{3,4} 
+\widetilde{\rm bII}^{3,6} +\widetilde{\rm bII}^{3,8} 
+\widetilde{\rm bII}^{13}_e +\widetilde{\rm bII}^{22}_o 
+\widetilde{\rm bII}^a_o, 
\\
\delta_1(\widetilde{\rm bI}^3_e) 
& = & \widetilde{\rm bII}^{2,3} +\widetilde{\rm bII}^{3,4} 
+\widetilde{\rm bII}^{3,6} +\widetilde{\rm bII}^{3,8} 
+\widetilde{\rm bII}^{13}_o +\widetilde{\rm bII}^{22}_e 
+\widetilde{\rm bII}^a_e, 
\\
\delta_1(\widetilde{\rm bI}^4_o) 
& = & \widetilde{\rm bII}^{2,4} +\widetilde{\rm bII}^{3,4} 
+\widetilde{\rm bII}^{4,6} +\widetilde{\rm bII}^{4,8} 
+\widetilde{\rm bII}^{13}_e +\widetilde{\rm bII}^{22}_o 
+\widetilde{\rm bII}^{23}_o +\widetilde{\rm bII}^{24}\\
& & \quad 
+\widetilde{\rm bII}^b_o +\widetilde{\rm bII}^f_o, 
\\
\delta_1(\widetilde{\rm bI}^4_e) 
& = & \widetilde{\rm bII}^{2,4} +\widetilde{\rm bII}^{3,4} 
+\widetilde{\rm bII}^{4,6} +\widetilde{\rm bII}^{4,8} 
+\widetilde{\rm bII}^{13}_o +\widetilde{\rm bII}^{22}_e 
+\widetilde{\rm bII}^{23}_e +\widetilde{\rm bII}^{24}\\
& & \quad 
+\widetilde{\rm bII}^b_e +\widetilde{\rm bII}^f_e, 
\\
\delta_1(\widetilde{\rm bI}^5_o) 
& = & \widetilde{\rm bII}^{2,5} +\widetilde{\rm bII}^{3,5} 
+\widetilde{\rm bII}^{4,5} +\widetilde{\rm bII}^{5,6} 
+\widetilde{\rm bII}^{5,8} +\widetilde{\rm bII}^{15} 
+\widetilde{\rm bII}^{23}_o +\widetilde{\rm bII}^{25}\\
& & \quad 
+\widetilde{\rm bII}^{30}_o 
+\widetilde{\rm bII}^{38}_o +\widetilde{\rm bII}^{e}_o, 
\\
\delta_1(\widetilde{\rm bI}^5_e) 
& = & \widetilde{\rm bII}^{2,5} +\widetilde{\rm bII}^{3,5} 
+\widetilde{\rm bII}^{4,5} +\widetilde{\rm bII}^{5,6} 
+\widetilde{\rm bII}^{5,8} +\widetilde{\rm bII}^{15} 
+\widetilde{\rm bII}^{23}_e +\widetilde{\rm bII}^{25}\\
& & \quad 
+\widetilde{\rm bII}^{30}_e 
+\widetilde{\rm bII}^{38}_e +\widetilde{\rm bII}^{e}_e, 
\\
\delta_1(\widetilde{\rm bI}^6_o) 
& = & \widetilde{\rm bII}^{2,6} +\widetilde{\rm bII}^{3,6} 
+\widetilde{\rm bII}^{4,6} +\widetilde{\rm bII}^{6,8} 
+\widetilde{\rm bII}^{c}_e +\widetilde{\rm bII}^{d}_o 
+\widetilde{\rm bII}^{e}_e +\widetilde{\rm bII}^{f}_e, 
\\
\delta_1(\widetilde{\rm bI}^6_e) 
& = & \widetilde{\rm bII}^{2,6} +\widetilde{\rm bII}^{3,6} 
+\widetilde{\rm bII}^{4,6} +\widetilde{\rm bII}^{6,8} 
+\widetilde{\rm bII}^{c}_o +\widetilde{\rm bII}^{d}_e 
+\widetilde{\rm bII}^{e}_o +\widetilde{\rm bII}^{f}_o, 
\\
\delta_1(\widetilde{\rm bI}^7_o) 
& = & \widetilde{\rm bII}^{2,7} +\widetilde{\rm bII}^{3,7} 
+\widetilde{\rm bII}^{4,7} +\widetilde{\rm bII}^{6,7} 
+\widetilde{\rm bII}^{7,8} +\widetilde{\rm bII}^{22} 
+\widetilde{\rm bII}^{23}_e +\widetilde{\rm bII}^{d}_o\\
& & \quad  
+\widetilde{\rm bII}^{f}_o, 
\\
\delta_1(\widetilde{\rm bI}^7_e) 
& = & \widetilde{\rm bII}^{2,7} +\widetilde{\rm bII}^{3,7} 
+\widetilde{\rm bII}^{4,7} +\widetilde{\rm bII}^{6,7} 
+\widetilde{\rm bII}^{7,8} +\widetilde{\rm bII}^{22} 
+\widetilde{\rm bII}^{23}_o +\widetilde{\rm bII}^{d}_e\\
& & \quad  
+\widetilde{\rm bII}^{f}_e, 
\\
\delta_1(\widetilde{\rm bI}^8_o) 
& = & \widetilde{\rm bII}^{2,8} +\widetilde{\rm bII}^{3,8} 
+\widetilde{\rm bII}^{4,8} +\widetilde{\rm bII}^{6,8} 
+\widetilde{\rm bII}^{23}_o +\widetilde{\rm bII}^{24} 
+\widetilde{\rm bII}^{c}_o +\widetilde{\rm bII}^{e}_o, 
\\
\delta_1(\widetilde{\rm bI}^8_e) 
& = & \widetilde{\rm bII}^{2,8} +\widetilde{\rm bII}^{3,8} 
+\widetilde{\rm bII}^{4,8} +\widetilde{\rm bII}^{6,8} 
+\widetilde{\rm bII}^{23}_e +\widetilde{\rm bII}^{24} 
+\widetilde{\rm bII}^{c}_e +\widetilde{\rm bII}^{e}_e, 
\\
\delta_1(\widetilde{\rm bI}^9_o) 
& = & \widetilde{\rm bII}^{2,9} +\widetilde{\rm bII}^{3,9} 
+\widetilde{\rm bII}^{4,9} +\widetilde{\rm bII}^{6,9} 
+\widetilde{\rm bII}^{8,9} +\widetilde{\rm bII}^{27} 
+\widetilde{\rm bII}^{35}_e +\widetilde{\rm bII}^{37}_o, 
\\
\delta_1(\widetilde{\rm bI}^9_e) 
& = & \widetilde{\rm bII}^{2,9} +\widetilde{\rm bII}^{3,9} 
+\widetilde{\rm bII}^{4,9} +\widetilde{\rm bII}^{6,9} 
+\widetilde{\rm bII}^{8,9} +\widetilde{\rm bII}^{27} 
+\widetilde{\rm bII}^{35}_o +\widetilde{\rm bII}^{37}_e, 
\\
\delta_1(\widetilde{\rm bI}^{10}_o) 
& = & \widetilde{\rm bII}^{2,10} +\widetilde{\rm bII}^{3,10} 
+\widetilde{\rm bII}^{4,10} +\widetilde{\rm bII}^{6,10} 
+\widetilde{\rm bII}^{8,10} +\widetilde{\rm bII}^{30}_e +\widetilde{\rm bII}^{32}\\
& & \quad 
+\widetilde{\rm bII}^{33} +\widetilde{\rm bII}^{35}_o 
+\widetilde{\rm bII}^{37}_o +\widetilde{\rm bII}^{38}_o 
+\widetilde{\rm bII}^{39}, 
\\
\delta_1(\widetilde{\rm bI}^{10}_e) 
& = & \widetilde{\rm bII}^{2,10} +\widetilde{\rm bII}^{3,10} 
+\widetilde{\rm bII}^{4,10} +\widetilde{\rm bII}^{6,10} 
+\widetilde{\rm bII}^{8,10} +\widetilde{\rm bII}^{30}_o +\widetilde{\rm bII}^{32}\\
& & \quad 
+\widetilde{\rm bII}^{33} +\widetilde{\rm bII}^{35}_e 
+\widetilde{\rm bII}^{37}_e +\widetilde{\rm bII}^{38}_e 
+\widetilde{\rm bII}^{39}, 
\end{eqnarray*}
where 
$\widetilde{\mathcal{F}}$ denotes
$\widetilde{\mathcal{F}}_o +\widetilde{\mathcal{F}}_e$. 

Then, by a straightforward calculation, we obtain the following. 

\begin{prop}
The cohomology groups of 
$
\mathcal{C}(b\mathcal{S}_{\mathrm{pr}}(3, 2), \rho_{3, 2}(2))
$
are described as follows:
\begin{itemize}
\item[$(1)$]
$H^0(b\mathcal{S}_{\mathrm{pr}}(3, 2), \rho_{3, 2}(2)) \cong \Z_2$, 
generated by $[\widetilde{\rm b0}_o +\widetilde{\rm b0}_e]$, 
\item[$(2)$]
$H^1(b\mathcal{S}_{\mathrm{pr}}(3, 2), \rho_{3, 2}(2)) \cong \Z_2 \oplus \Z_2$, 
generated by 
\begin{eqnarray*}
\beta & = & [\widetilde{\rm bI}^6 +\widetilde{\rm bI}^7 +\widetilde{\rm bI}^8] 
=
[\widetilde{\rm bI}^2 +\widetilde{\rm bI}^3 
+\widetilde{\rm bI}^4 +\widetilde{\rm bI}^7], 
\\
\gamma 
& = & [\widetilde{\rm bI}^2_o +\widetilde{\rm bI}^3_e +\widetilde{\rm bI}^4_e 
+\widetilde{\rm bI}^6_o +\widetilde{\rm bI}^8_e] 
= 
[\widetilde{\rm bI}^2_e +\widetilde{\rm bI}^3_o +\widetilde{\rm bI}^4_o 
+\widetilde{\rm bI}^6_e +\widetilde{\rm bI}^8_o]. 
\end{eqnarray*}
\end{itemize}
\end{prop}

Note that the ranks of $C^i(b\mathcal{S}_{\mathrm{pr}}(3, 2), \rho_{3, 2}(2))$, 
$i =0, 1, 2$, are equal to 
$2$, $18$ and $160$, respectively. 

Let 
\[
{s_\kappa}_\ast \co 
H^\kappa(b\mathcal{S}_{pr}(3, 2), \rho_{3, 2}(2)) \to 
H^\kappa(b\mathcal{S}_{pr}(2, 1), \rho_{2, 1}(2))
\]
be the homomorphism induced by suspension $s_{\kappa}$. 
Note that $s_\kappa(\mathcal{F}_*) =\mathcal{F}_*$ 
for $\kappa =0, 1$ and for each $\rho_{3, 2}(2)$-class $\mathcal{F}_*$
of codimension $\kappa$
(see Remark~\ref{rem:Morse}).
Note also that $s_\kappa =0$ for $\kappa = 2$. 
Then, 
Lemma~\ref{coexistence2to1} shows that $s_{1 \ast} \beta$ 
induces a trivial $b\mathcal{S}_{\mathrm{pr}}$-cobordism invariant 
for stable Morse functions 
$f\co V \to W$ on compact surfaces $V$ with boundary
into $W = \R$ or $S^1$. 
Furthermore, by using the same method as in \cite[Lemma~14.1]{Saeki04}, 
we can show that 
$s_{1 \ast} \gamma$ also induces a trivial 
$b\mathcal{S}_{\mathrm{pr}}$-cobordism invariant.
Therefore, unfortunately, the above cohomology information
does not provide us of a new cobordism invariant.

Let us now consider a certain restricted class of
stable maps. 
For a positive integer $n$, let
$\mathcal{AS}_{\mathrm{pr}}(n, n-1)$
be the set of fibers for proper admissible $C^0$ stable Thom maps
of $n$-dimensional manifolds with boundary into $(n-1)$-dimensional
manifolds without boundary, where a $C^0$ stable map $f\co M \to N$ 
of an $n$-dimensional manifold with boundary
into an $(n -1)$-dimensional manifold without boundary 
is \emph{admissible} 
if it is a submersion on a neighborhood of $\partial M$. 
In other words, for a stable map $f\co M \to N$ 
of a $3$-dimensional manifold with boundary
into a surface without boundary 
such a map can be characterized
as a stable map without
definite $\Sigma^{2, 0}_{1, 0}$ points nor 
indefinite $\Sigma^{2, 0}_{1, 0}$ points. 
Note that stable Morse
functions on compact surfaces and their
suspensions are always admissible.

Furthermore, set
$$\mathcal{AS}_{pr} = \bigcup_{n=1}^\infty
\mathcal{AS}_{\mathrm{pr}}(n, n-1).$$
Note that the above set together with the equivalence
relation induced by $\rho(2)$, which we still denote
by $\rho(2)$ by abuse of notation, satisfy conditions
(a)--(e) mentioned before.

Then, a straightforward calculation shows
the following.

\begin{prop}
The cohomology groups of the universal complex
$$
\mathcal{C}(\mathcal{AS}_{pr}(3, 2), \rho_{3, 2}(2))
$$
for admissible stable maps of $3$-manifolds with
boundary to surfaces without boundary with
respect to the $C^0$ equivalence modulo two
regular fiber components
are described as follows:
\begin{itemize}
\item[$(1)$]
$H^0(\mathcal{AS}_{\mathrm{pr}}(3, 2), \rho_{3, 2}(2)) \cong \Z_2$,
generated by $[\widetilde{\rm b0}_o +\widetilde{\rm b0}_e]$, 
\item[$(2)$]
$H^1(\mathcal{AS}_{\mathrm{pr}}(3, 2), \rho_{3, 2}(2)) 
\cong \Z_2 \oplus \Z_2 \oplus \Z_2$, 
generated by
\begin{eqnarray*}
\alpha & = & [\widetilde{\rm bI}^2 +\widetilde{\rm bI}^3 +\widetilde{\rm bI}^4 
+\widetilde{\rm bI}^5 +\widetilde{\rm bI}^9 +\widetilde{\rm bI}^{10}], 
\\
\beta & = & [\widetilde{\rm bI}^2 +\widetilde{\rm bI}^3 +\widetilde{\rm bI}^4 
+\widetilde{\rm bI}^7] 
= [\widetilde{\rm bI}^6 +\widetilde{\rm bI}^7 +\widetilde{\rm bI}^8],
\\ 
\gamma 
& = & [\widetilde{\rm bI}^2_o +\widetilde{\rm bI}^3_e +\widetilde{\rm bI}^4_e 
+\widetilde{\rm bI}^6_o +\widetilde{\rm bI}^8_e] 
= 
[\widetilde{\rm bI}^2_e +\widetilde{\rm bI}^3_o +\widetilde{\rm bI}^4_o 
+\widetilde{\rm bI}^6_e +\widetilde{\rm bI}^8_o]. 
\end{eqnarray*}
\end{itemize}
\end{prop}

Note that the ranks of $C^i(\mathcal{AS}_{\mathrm{pr}}(3, 2), \rho_{3, 2}(2))$, 
$i =0, 1, 2$, are equal to 
$2$, $18$ and $154$, respectively. 

Let 
\[
s_{\kappa \ast} \co 
H^\kappa(\mathcal{AS}_{\mathrm{pr}}(3, 2), \rho_{3, 2}(2)) \to 
H^\kappa(\mathcal{AS}_{\mathrm{pr}}(2, 1), \rho_{2, 1}(2))
\]
be the homomorphism induced by suspension $s_{\kappa}$. 
Then, we obtain the following cobordism invariants. 

\begin{cor}
Let $W$ be the real line $\R$ or the circle $S^1$.

$(1)$ The cohomology class $s_{1 \ast} \alpha$ induces a non-trivial 
$\mathcal{AS}_{\mathrm{pr}}$-cobordism invariant for
stable Morse functions 
$f\co V \to W$ on compact surfaces $V$ with boundary
into $W$. 

$(2)$ The cohomology classes
$s_{1 \ast} \beta$ and $s_{1 \ast} \gamma$ induce 
trivial $\mathcal{AS}_{\mathrm{pr}}$-cobordism invariants for
stable Morse functions 
$f\co V \to W$ of compact surfaces $V$ with boundary
into $W$. 
\end{cor}

\begin{proof}
A stable Morse function $f\co D^2 \to \R$ given by
the height function as depicted in  
Figure~\ref{Example2}
shows that the cobordism invariant $s_{1 \ast} \alpha$ is non-trivial, 
where $D^2$ denotes the $2$-dimensional disk. 
The same construction works also for $W = S^1$.

\begin{figure}[t]
\psfrag{D}{$D^2 =$}
\psfrag{f}{$f$}
\includegraphics[scale=.8]{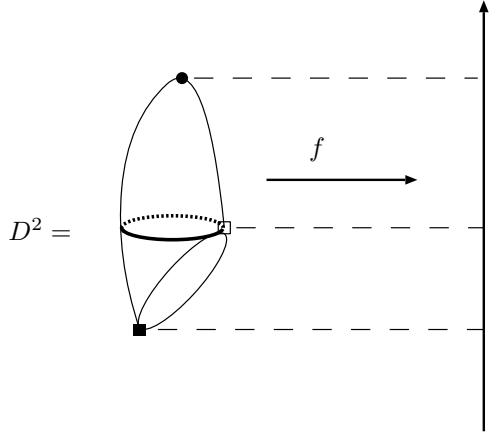}
\caption{A stable Morse function on the disk}
\label{Example2}
\end{figure}

On the other hand,
Remark~\ref{coexistence2to1} shows that the invariant
associated with $s_{1 \ast} \beta$ 
is trivial. 
The triviality of the invariant associated with
$s_{1 \ast} \gamma$ is proved by using the same argument
as in \cite[Lemma~14.1]{Saeki04}, so we omit the details here. 
\end{proof}

\begin{rem}
The above corollary shows that even if we take a non-trivial 
cohomology class of the universal complex, 
the corresponding cohomology class in the target manifold can be 
always trivial. 
\end{rem}

\begin{rem}
We can also consider 
co-orientable singular fibers
in the sense of \cite[Definition~10.5]{Saeki04} and
calculate the $1$st cohomology group of the
universal complex
of co-orientable fibers 
for proper admissible stable
maps of $3$-dimensional manifolds with boundary into surfaces
without boundary with respect to the $C^0$ equivalence
modulo two regular fibers. Unfortunately, the cobordism
invariants that we get are all
trivial, so we omit the details here
(for details, see~\cite{SY15}).
\end{rem}

\section*{Acknowledgment}\label{ack}
The authors would like to express their sincere gratitude to
Shigeo Takahashi, Daisuke Sakurai,
Hsiang-Yun Wu, Keisuke Kikuchi and
Hamish Carr for stimulating discussions and for
posing interesting questions.
The authors would also like to thank the referee
for very important comments which drastically improved the
presentation of the paper.
The first author has been supported in part by JSPS KAKENHI Grant Number 
23244008, 23654028, 25540041, 15K13438.
The second author has been supported in part by JSPS KAKENHI Grant Number 23654028,
15K13438.



\begin{thebibliography}{99}


\bibitem{Damon77}J.~Damon, 
\textit{The relation between $C^\infty$ and topological stability}, 
Bol.\ Soc.\ Brasil.\ Mat.\ \textbf{8} (1977), 1--38. 

\bibitem{duPlessisWall95}A.~du Plessis and T.~Wall, 
\textit{The geometry of topological stability}, 
London Math.\ Soc.\ Monographs, New Series 9, Oxford Science Publ., 
The Clarendon Press, Oxford Univ.\ Press, New York, 1995. 


\bibitem{GWDL}C.G.~Gibson, K.~Wirthm\"uller, A.A.~du Plessis and
E.J.N.~Looijenga, \textit{Topological stability of smooth mappings},
Lecture Notes in Math., Vol.~552, Springer-Verlag, Berlin,
New York, 1976.

\bibitem{GolubitskyGuillemin}M.~Golubitsky and V.~Guillemin, 
\textit{Stable mappings and their singularities}, 
Grad.\ Texts in Math., Vol.~14, Springer, New York-Heidelberg, 1973.





\bibitem{MartinsNabarro13}
L.F.~Martins and A.C.~Nabarro, 
\textit{Projections of hypersurfaces in $\R^4$ with boundary to planes}, 
Glasgow Math.\ J.\ \textbf{56} (2014), 149--167.


\bibitem{Ohmoto}T.~Ohmoto, \textit{Vassiliev complex for contact classes of 
real smooth map-germs}, Rep.\ Fac.\ Sci.\ Kagoshima Univ.\ 
Math.\ Phys.\ Chem.\ \textbf{27} (1994), 1--12. 


\bibitem{Saeki04}O.~Saeki, 
\textit{Topology of singular fibers of differentiable maps}, 
Lecture Notes in Math., Vol.~1854, Springer, Berlin, 2004.

\bibitem{Saeki06}O.~Saeki, 
\textit{Cobordism of Morse functions on surfaces, the universal complex 
of singular fibers and their application to map germs}, 
Algebr.\ Geom.\ Topol.\  \textbf{6} (2006), 539--572. 

\bibitem{ST2014}O.~Saeki and S.~Takahashi, \textit{Visual data
mining based on differential topology: a survey}, 
Pacific Journal of Mathematics for Industry \textbf{6} (2014),
Article 4,
DOI: 10.1186/s40736-014-0004-y.

\bibitem{STSWKCDY2013}O.~Saeki, S.~Takahashi, D.~Sakurai,
Hsiang-Yun Wu, K.~Kikuchi, H.~Carr, D.~Duke and T. Yamamoto,
\textit{Visualizing multivariate data using singularity theory},
in ``The Impact of Applications on Mathematics'', Proceedings of 
Forum ``Math-for-Industry'' 2013, 
Mathematics for Industry, Volume 1, 2014, pp.~51--65,
Springer.

\bibitem{SaekiY06}O.~Saeki and T.~Yamamoto, 
\textit{
Singular fibers of stable maps and signatures of $4$-manifolds}, 
Geom.\ Topol.\ \textbf{10} (2006), 359--399. 

\bibitem{SaekiY07}O.~Saeki and T.~Yamamoto, 
\textit{Singular fibers and characteristic classes}, 
Topology Appl.\ \textbf{155} (2007), 112--120.

\bibitem{SY15}O.~Saeki and T.~Yamamoto, 
\textit{Co-orientable singular fibers of stable maps of 
$3$-manifolds with boundary into surfaces}, 
to appear in RIMS K\^oky\^uroku,
Research Institute for Mathematical
Sciences, Kyoto University.

\bibitem{Shibata00}N.~Shibata, 
\textit{On non-singular stable maps of $3$-manifolds with 
boundary into the plane}, 
Hiroshima Math.\ J.\ \textbf{39} (2000), 415--435. 


\bibitem{V}V.A.~Vassilyev, \textit{Lagrange and Legendre characteristic 
classes}, Translated from the Russian, Advanced Studies in 
Contemporary Mathematics, Vol.~3, Gordon and Breach Science 
Publishers, New York, 1988. 



\bibitem{Y06}T.~Yamamoto, 
\textit{Classification of singular fibres of 
stable maps of $4$-manifolds into 
$3$-manifolds and its applications}, 
J.\ Math.\ Soc.\ Japan \textbf{58} (2006), 721--742. 

\bibitem{Y07}T.~Yamamoto, 
\textit{Euler number formulas in terms of singular fibers of stable maps}, 
in ``Real and complex singularities'', 
pp.~427--457, 
World Sci.\ Publ., Hackensack, NJ, 2007. 

\bibitem{Y08}T.~Yamamoto, 
\textit{Singular fibers of two-colored maps and cobordism invariants}, 
Pacific J.\ Math.\ \textbf{234} (2008), 379--398.

\end{thebibliography}
\end{document}